\documentclass[a4paper, 11pt]{amsart}

\usepackage{amsmath,amssymb,amsthm,url,scalerel}
\usepackage[utf8]{inputenc}
\usepackage[T1]{fontenc}
\usepackage{libertine}
\usepackage[libertine,cmintegrals,cmbraces]{newtxmath}
\usepackage{verbatim}
\usepackage[shortlabels]{enumitem}
\usepackage{stmaryrd}
\usepackage{mathtools}
\usepackage{microtype}
\usepackage{tabto}
\usepackage{xcolor}
\usepackage{euscript}
\usepackage[all]{xy}
\usepackage{tikz}
\usepackage{tikz-cd}
\usepackage{nicefrac}
\usepackage{tabularx}
\usepackage{dsfont}
\usepackage{todonotes}
\usepackage{extarrows}
\usepackage{quiver}
\usepackage{mathdots}

\numberwithin{equation}{section} % or whatever you prefer

\AtBeginDocument{%
   \def\MR#1{}
}

\usepackage[pagebackref]{hyperref}
  
\hypersetup{%
  bookmarksnumbered=true,%%
  colorlinks=true,%
  linkcolor=blue,%
  citecolor=blue,%
  filecolor=blue,%
  menucolor=blue,%
  urlcolor=blue,%
  pdfnewwindow=true,%
  pdfstartview=FitBH}

\usepackage{xcolor}
\definecolor{seagreen}{RGB}{46,139,87}
\definecolor{maroon}{RGB}{128,0,0}
\definecolor{darkviolet}{RGB}{148,0,211}
\definecolor{twelve}{RGB}{100,100,170}
\definecolor{thirteen}{RGB}{100,150,50}
\definecolor{fourteen}{RGB}{200,0,0}
\definecolor{fifteen}{RGB}{0,200,0}
\definecolor{sixteen}{RGB}{0,0,200}
\definecolor{seventeen}{RGB}{200,0,200}
\definecolor{eighteen}{RGB}{0,200,200}

\newcommand{\C}{C_2}
\newcommand{\Rdelta}{{\mathbb{R}^\delta}}
\newcommand{\R}{\mathbb{R}}
\newcommand{\RC}{\mathbb{R}[\C]}
\newcommand{\CL}{C_2\mathcal{L}}
\newcommand{\Jmor}[3]{C_2\mathcal{J}_{#1}(#2,#3)}
\newcommand{\Gmor}[3]{C_2\mathcal{\gamma}_{#1}(#2,#3)}
\newcommand{\CTop}{\C\TTop}
\newcommand{\COTop}{(O(p,q)\rtimes\C)\TTop}
\newcommand{\TTop}{\text{Top}}
\newcommand{\Jpq}{C_2\mathcal{J}_{p,q}}
\newcommand{\Jzero}{C_2\mathcal{J}_{0,0}}
\newcommand{\Jzerobar}{\overline{C_2\mathcal{J}_{0,0}}}
\newcommand{\Epq}{C_2\mathcal{E}_{p,q}}
\newcommand{\Ezero}{C_2\mathcal{E}_{0,0}}
\newcommand{\LL}{ \mathcal{L} }
\newcommand{\OEpq}{O(p,q)C_2\mathcal{E}_{p,q}}
\newcommand{\OEpql}{O(p,q)C_2\mathcal{E}_{p,q}^{l}}
\newcommand{\CJ}[2]{C_2\mathcal{J}_{#1,#2}}
\newcommand{\CE}[2]{C_2\mathcal{E}_{#1,#2}}
\newcommand{\OCE}[2]{O(#1,#2)C_2\mathcal{E}_{#1,#2}}
\newcommand{\taupq}{\tau_{p,q}}
\newcommand{\Tpq}{T_{p,q}}
\newcommand{\pqs}{(p,q)\mathbb{S}}
\newcommand{\Ilevel}{I_{level}}
\newcommand{\Jlevel}{J_{level}}
\newcommand{\Jstable}{J_{stable}}

\DeclareMathOperator{\ind}{ind}
\DeclareMathOperator{\Nat}{Nat}
\DeclareMathOperator{\res}{res}
\DeclareMathOperator{\colim}{colim}
\DeclareMathOperator{\hocolim}{hocolim}

\DeclareMathOperator{\hofibre}{hofibre}
\DeclareMathOperator{\holim}{holim}
\DeclareMathOperator{\Dim}{dim}
\DeclareMathOperator{\Min}{min}
\DeclareMathOperator{\conn}{conn}
\DeclareMathOperator{\Ho}{Ho}
\DeclareMathOperator{\CI}{CI}
\DeclareMathOperator{\proj}{proj}
\DeclareMathOperator{\poly}{-poly-}
\DeclareMathOperator{\homog}{-homog-}
\DeclareMathOperator{\id}{id}
\DeclareMathOperator{\Id}{Id}

\DeclareMathOperator{\Fun}{Fun}

  \newcommand{\adjunction}[4]{
\xymatrix{
#1:#2 \ar@<.5ex>[r] &
\ar@<.5ex>[l] #3:#4
}}

\newtheorem{xxthm}{Theorem}
\newtheorem{xxprop}[xxthm]{Proposition}

\newtheorem{thm}{Theorem}[subsection]
\newtheorem{prop}[thm]{Proposition}
\newtheorem{lem}[thm]{Lemma}
\newtheorem{cor}[thm]{Corollary}

\newtheorem*{thm*}{Theorem}
\newtheorem*{ex*}{Example}
\newtheorem*{exs*}{Examples}

\theoremstyle{definition}
\newtheorem{definition}[thm]{Definition}
\newtheorem{ex}[thm]{Example}

\newtheorem{rem}[thm]{Remark}

\begin{document}

\title{$\C$ - Equivariant Orthogonal Calculus}
\author{Emel Yavuz}
\date{\today}
%\thanks{}

\begin{abstract}\fontsize{10pt}{13pt}\selectfont
In this paper, we construct a version of orthogonal calculus for functors from $\C$-representations to $\C$-spaces, where $\C$ is the cyclic group of order 2. For example, the functor $BO(-)$, that sends a $\C$-representation to the classifying space of its orthogonal group, which has a $\C$-action induced by the action on the $\C$-representation. We obtain a bigraded sequence of approximations to such a functor, and via a zig-zag of Quillen equivalences, we prove that the homotopy fibres of maps between approximations are fully determined by orthogonal spectra with a genuine action of $\C$ and a naive action of the orthogonal group $O(p,q):=O(\mathbb{R}^{p+q\delta})$. 
\end{abstract}

\maketitle

\section{Introduction}

Given a real function $f:\R\rightarrow\R$, Taylor's theorem describes how $f$ can be approximated by a sequence of polynomial functions, which are built using the derivatives of $f$. As a result, in order to study complicated functions, it suffices to study polynomial functions, which are well understood. 

Many objects studied in algebraic topology can be realised as functors. Functor calculus is a method by which one can approximate a given functor by a sequence of functors with `nice' properties, which we call polynomial functors. The resulting sequence of functors is a similar concept to that of a Postnikov tower. Polynomial functors have properties that mimic those of the polynomial functions used in differential calculus. For example, an $n$-polynomial functor is also $(n+1)$-polynomial and the $(n+1)$-st derivative of an $n$-polynomial functor is trivial. 

There are many different branches of functor calculus designed to study different categories of functors. Goodwillie calculus, originally constructed by Goodwillie \cite{Goo90, Goo92, Goo03}, is used to study endofunctors on the category of topological spaces. The fibres of the tower produced by Goodwillie calculus are classified by spectra with an action of the symmetric group $\Sigma_n$. The main focus of this paper is the orthogonal homotopy calculus first constructed by Weiss in \cite{Wei95}. It is the branch of functor calculus involving the study of functors from the category of finite dimensional real vector spaces to the category of pointed topological spaces. The tower for a functor $F$ produced by orthogonal calculus looks as follows,   
% https://q.uiver.app/#q=WzAsNyxbMiwzLCJGKFxcUl5cXGluZnR5KSJdLFsyLDIsIlRfMUYoVikiXSxbMiwxLCJUXzJGKFYpIl0sWzIsMCwiXFx2ZG90cyJdLFszLDIsIlxcT21lZ2FeXFxpbmZ0eSBcXGxlZnRbXFxsZWZ0KFNeezFWfSBcXHdlZGdlIFxcVGhldGFfRl4xXFxyaWdodClfe2hPKDEpfVxccmlnaHRdIl0sWzAsMywiRihWKSJdLFszLDEsIlxcT21lZ2FeXFxpbmZ0eSBcXGxlZnRbXFxsZWZ0KFNeezJWfSBcXHdlZGdlIFxcVGhldGFfRl4yXFxyaWdodClfe2hPKDIpfVxccmlnaHRdIl0sWzMsMl0sWzIsMV0sWzEsMF0sWzQsMV0sWzUsMF0sWzYsMl0sWzUsMSwiIiwwLHsiY3VydmUiOi0xfV0sWzUsMiwiIiwwLHsiY3VydmUiOi0yfV1d
\[\begin{tikzcd}
	&& \vdots \\
	&& {T_2F(V)} & {\Omega^\infty \left[\left(S^{2V} \wedge \Theta_F^2\right)_{hO(2)}\right]} \\
	&& {T_1F(V)} & {\Omega^\infty \left[\left(S^{1V} \wedge \Theta_F^1\right)_{hO(1)}\right]} \\
	{F(V)} && {F(\R^\infty)}
	\arrow[from=1-3, to=2-3]
	\arrow[from=2-3, to=3-3]
	\arrow[from=3-3, to=4-3]
	\arrow[from=3-4, to=3-3]
	\arrow[from=4-1, to=4-3]
	\arrow[from=2-4, to=2-3]
	\arrow[curve={height=-6pt}, from=4-1, to=3-3]
	\arrow[curve={height=-12pt}, from=4-1, to=2-3]
\end{tikzcd}\]
where each functor $T_nF$ is $n$-polynomial and the $n^\text{th}$ fibre of the tower is fully determined by an orthogonal spectrum $\Theta_F^n$ with an action of the orthogonal group $O(n)$.

Classic examples of functors studied using orthogonal calculus include:
\begin{itemize}
    \item $BO(-):V\mapsto BO(V)$
    \item $B\Top(-):V\mapsto B\Top(V)$
    \item $B\text{Diff}^b(-):V\mapsto B\text{Diff}^b(M\times V)$
\end{itemize}
where $BO(V)$ is the classifying space of the space of linear isometries on V, $B\Top(V)$ is the classifying space of the space of homeomorphisms on V and, for a smooth compact manifold $M$, $B\text{Diff}^b(M\times V)$ is the classifying space of the space of bounded diffeomorphisms on $M\times V$. 

There exist functors, similar to those above, that have group actions. For example, the functor $BO(-):V\rightarrow BO(V)$ that sends a $G$-representation to its classifying space. As such, there is a natural motivation to construct functor calculi that study functors with a group action. An equivariant orthogonal calculus of this type could have applications in many different areas, such as the study of equivariant diffeomorphisms of $G$-manifolds. Extensive research focused on equivariance in the Goodwillie calculus setting has been carried out by Dotto \cite{Dot16a,Dot16b,Dot17} and Dotto and Moi \cite{DM16}. 

\subsection*{Main results and organisation}

The $\C$-equivariant orthogonal calculus gives a method for studying functors from $\C$-representations to the category of $\C$-spaces. For example, the functor $$BO(-):V\mapsto BO(V),$$where $BO(V)$ is the classifying space of the orthogonal group $O(V)$, which has a $\C$-action induced by the action on the $\C$-representation $V$.

We index the $\C$-equivariant orthogonal calculus over a category of $\C$-representations, which are isomorphic to representations of the form $\R^{p+q\delta}=\R^p\oplus \R^{q\delta}$, where $\Rdelta$ is the sign $\C$-representation. The main result, Theorem \ref{thm A}, is the classification of $(p,q)$-homogeneous functors (defined in Section \ref{sec: homog functors}), which are the $\C$-equivariant analogue of functors that are $n$-homogeneous in orthogonal calculus (functors with polynomial approximations concentrated in degree $n$). We show that $(p,q)$-homogeneous functors are fully determined by genuine orthogonal $\C$-spectra with an action of the orthogonal group $O(p,q):=O(\R^{p+q\delta})$, which has a specified $\C$-action given in Definition \ref{def:O(p,q) and matrix A}. That is, orthogonal spectra with a genuine action of $\C$ and a naive action of $O(p,q)$, denoted $\C Sp^{\mathcal{O}}[O(p,q)]$. In this way, we get a richer equivariant structure compared to that of calculus with reality \cite{Tag22real}, in which the classification is in terms of spectra with a naive action of $\C\ltimes U(n)$.  

\begin{xxthm}[{Theorem \ref{weissclassification}}]\label{thm A}
Let $p,q\geq 1$. If $F$ is a $(p,q)$-homogeneous functor, then $F$ is objectwise weakly equivalent to 
\begin{equation*}
    V\mapsto \Omega^\infty[(S^{(p,q)V}\wedge\Theta_F^{p,q})_{hO(p,q)}],
\end{equation*}
where $\Theta_F^{p,q}\in \C Sp^{\mathcal{O}}[O(p,q)]$ and $(-)_{hO(p,q)}$ denotes homotopy orbits. 

Conversely, every functor of the form 
\begin{equation*}
    V\mapsto \Omega^\infty[(S^{(p,q)V}\wedge\Theta)_{hO(p,q)}],
\end{equation*}
where $\Theta\in \C Sp^{\mathcal{O}}[O(p,q)]$, is $(p,q)$-homogeneous. 
\end{xxthm}

The classification can alternatively be stated in terms of the following zig-zag of Quillen equivalences between the calculus input category $\Ezero$ (of functors from the category of finite dimensional $\C$-representations with inner product and linear isometries to $\C\TTop_*$) and the category of genuine orthogonal $\C$-spectra with an action of $O(p,q)$. 

\begin{xxthm}[{Theorem \ref{boclassification} and Theorem \ref{QEstabletospectra}}]\label{thm B}
For all $p,q\geq 1$, there exist Quillen equivalences
\[\begin{tikzcd}
	{(p,q)\homog C_2\mathcal{E}_{0,0}} && {O(p,q)C_2\mathcal{E}_{p,q}^s} && {C_2Sp^O[O(p,q)]}
	\arrow["{\ind_{0,0}^{p,q}\varepsilon^*}"', shift right=2, from=1-1, to=1-3]
	\arrow["{\res_{0,0}^{p,q}/O(p,q)}"', shift right=2, from=1-3, to=1-1]
	\arrow["{(\alpha_{p,q})_!}", shift left=2, from=1-3, to=1-5]
	\arrow["{\alpha_{p,q}^*}", shift left=2, from=1-5, to=1-3]
\end{tikzcd}\]
\end{xxthm}
Here $(p,q)\homog C_2\mathcal{E}_{0,0}$ denotes the $(p,q)$-homogeneous model structure on the input category. This model structure captures the structure of $(p,q)$-homogeneous functors, in that the cofibrant-fibrant objects are exactly the projectively cofibrant $(p,q)$-homogeneous functors. This model structure is detailed in Section \ref{sec:homog ms}. The zig-zag of equivalences is made up of two Quillen equivalences. Differentiation (also called induction) forms a Quillen functor from the $(p,q)$-homogeneous model structure to an intermediate category of functors $O(p,q)\C\mathcal{E}_{p,q}^s$, which is in turn Quillen equivalent to the category of genuine orthogonal $\C$-spectra with an action of $O(p,q)$.

In comparison to the underlying calculus which is indexed over $\mathbb{N}$, $\C$-equivariant orthogonal calculus is bi-indexed over $\mathbb{N}\times\mathbb{N}$. As a result, we can define differentiation in two directions (the $p$-direction and the $q$-direction). These different derivatives act like partial derivatives in differential calculus; in particular, they commute. 

A key difference between the underlying and $\C$-equivariant orthogonal calculi is an indexing shift, caused by this bi-indexing. In particular, $\tau_n$ in the underlying calculus is defined using the poset of non-zero subspaces $\{0\neq U\subseteq \R^{n+1}\}$ and $\taupq$ in the $\C$-calculus is defined using the poset of non-zero subspaces $\{0\neq U\subseteq \R^{p+q\delta}\}$. To keep notation consistent, the author introduced the new term strongly $(p,q)$-polynomial, see Definition \ref{def: polynomial}. A functor in the input category for $\C$-equivariant orthogonal calculus is then called $(p,q)$-polynomial if and only if it is both strongly $(p+1,q)$-polynomial and strongly $(p,q+1)$-polynomial. In particular, we define the strongly $(p,q)$-polynomial approximation functor $\Tpq$, and the $(p,q)$-polynomial approximation functor is the composition $T_{p+1,q}T_{p,q+1}$. A functor $X$ is $(p,q)$-homogeneous if it is $(p,q)$-polynomial and the strongly $(p,q)$-polynomial approximation $T_{p,q}X$ is trivial. 

\begin{xxthm}[{Theorem \ref{thm: DYpq is homog}}]
The homotopy fibres of the maps 
\begin{equation*}
    T_{p+1,q}T_{p,q+1}F\rightarrow \Tpq F
\end{equation*}
are $(p,q)$-homogeneous, and can therefore be described in terms of genuine orthogonal $\C$-spectra with an action of $O(p,q)$, by the classification given in Theorem \ref{thm A}. 
\end{xxthm}

A key result of the calculus, which makes the classification work, is the existence of the following $\C$-homotopy cofibre sequences. These cofibre sequences tell us that derivatives in $\C$-equivariant orthogonal calculus are well behaved. The notation $\C\mathcal{J}_{p,q}:=\C\mathcal{J}_{\mathbb{R}^{p,q}}$ denotes the $(p,q)$-th jet category whose objects are $\C$-representations and morphisms are given by $\Jmor{p,q}{U}{V}$, which is a $\C$-space (see Definition \ref{def: p,q-jet cat}).

\begin{xxprop}[{Proposition \ref{cofibseq}}]\label{prop D}
For all $U,V,W$ in $\Jzero$, the homotopy cofibre of the map
\begin{equation*}
\Jmor{W}{U\oplus \mathbb{R}^\alpha}{V}\wedge S^{W\otimes \mathbb{R}^\alpha} \rightarrow \Jmor{W}{U}{V}
\end{equation*}
is $\C$-homeomorphic to $\Jmor{W\oplus \mathbb{R}^\alpha}{U}{V}$, where $\mathbb{R}^\alpha$ is either the trivial or sign representation of $\C$, and $S^{W\otimes \mathbb{R}^\alpha}$ denotes the one point compactification of $W\otimes \mathbb{R}^\alpha$.
\end{xxprop}  

It is very difficult to produce variations of orthogonal calculus due to the nature of its construction. Two successful variations are the unitary calculus and calculus with reality constructed by Taggart in \cite{Tag22unit, Tag22real}. In these calculi, real vector spaces are replaced by complex vector spaces, and in the calculus with reality one takes into consideration the $\C$-action on complex vector spaces given by complex conjugation. The fibres of the towers produced are classified by spectra with an action of the unitary group $U(n)$ for unitary calculus and spectra with an action of $\C\ltimes U(n)$ for calculus with reality. Taggart's calculus with reality provides a great insight of what can be expected from a genuine $C_2$-equivariant orthogonal calculus, and a number of proofs in this paper were inspired by the extensions of Taggart.  

The main difficulty in generalising the $\C$-equivariant orthogonal calculus to $G$-equivariant orthogonal calculus, for an arbitrary group $G$, is the cofibre sequence of Proposition \ref{prop D}. The cofibre sequence holds when $\mathbb{R}^\alpha$ is either the trivial or sign representation of $\C$, since they are the one-dimensional irreducible $\C$-representations. Replacing $\mathbb{R}^\alpha$ with a representation of dimension greater than one would require an iterated cofibre sequence. In particular, this indicates that replicating this type of cofibre sequence for a general group $G$ could be difficult, since $G$ might have irreducible representations with dimension greater than one. 

Work of Bhattacharya and Hu \cite{BH2024} produces analogous results for the case of a general finite group, but still requires that their derivatives be taken along one-dimensional representations (see \cite[Theorem 4.1]{BH2024}). 

This difficulty can be avoided if one restricts to abelian groups and the complex setting, since every irreducible representation of a finite abelian group over $\mathbb{C}$ is one-dimensional. Therefore, it should be possible to construct a $G$-equivariant unitary calculus, for $G$ a finite abelian group. It might be possible to then recover $G$-equivariant orthogonal calculus via a complexification functor, similar to that used by Taggart \cite{Tag21}, however one should be careful to check that this preserves the $G$-equivariance.

\subsection*{Notation and conventions}

Throughout this paper we will use $\TTop_*$ to denote the category of pointed compact generated weak Hausdorff spaces with the Quillen model structure. We denote the category of pointed $\C$-spaces and $\C$-equivariant maps by $\C\TTop_*$. We endow $\C\TTop_*$ with the fine model structure, which is cofibrantly generated by the sets of generating cofibrations and acyclic cofibrations $I_{\C}$ and $J_{\C}$ respectively.

We denote the trivial element of $\C$ by $e$ and the non-trivial element by $\sigma$. The one-dimensional trivial and sign $\C$-representations are denoted by $\R$ and $\Rdelta$ respectively. We assume all $\C$-representations are equipped with an inner product that respects the $\C$-action, and use the notation $\R^{p,q}$ to denote the representation $\R^p\oplus\R^{q\delta}$. For a $\C$-representation $V$, we denote the tensor product $\R^{p,q}\otimes V$ by $(p,q)V$. Finally, we denote the semi-direct product of $\C$ with $O(p,q)$ by $O(p,q) \rtimes \C$.

\subsection*{Acknowledgements}

This paper is a condensed version of the authors Ph.D. thesis \cite{Yav24}. The author would like to express sincere gratitude to their supervisor, David Barnes, whose support, guidance and expertise were instrumental in the completion of this thesis project. The author also thanks Niall Taggart, whose knowledge, advice and time have been invaluable.

\section{Equivariant functor categories}\label{sec: equiv functor cats}
The primary objects of study in calculus of real functions are derivatives. In orthogonal calculus, one constructs derivatives of input functors via combinations of restriction functors and the inflation-orbit change-of-group functors for spaces. These derivatives play a key role in the classification of $n$-homogeneous functors, as they form part of the zig-zag of Quillen Equivalences used to derive the classification Theorem. We extend this notion to the $\C$-equivariant setting by defining new functor categories and adjunctions analogous to those used in the underlying calculus. We begin by choosing a new indexing category that will induce the $\C$-actions used throughout the calculus. 

\subsection{The input functors}\label{sect: input functors}
Orthogonal calculus is indexed on the universe $\mathbb{R}^\infty$, which makes the input category of functors enriched over topological spaces. To guarantee that the category of input functors for $\C$-equivariant orthogonal calculus is enriched over pointed $\C$-spaces, we must specify a new universe which is closed under $\C$-action. 

The regular representation of $\C=\{e,\sigma\}$ is defined as the following vector space\begin{equation*}
    \RC=\{\lambda_1 \underline{e}+\lambda_2 \underline{\sigma}:\lambda_1,\lambda_2\in\R\}
\end{equation*}
with basis elements $\underline{e}, \underline{\sigma}$. To better understand the $\C$-action on the vector space $\RC$, we can decompose \begin{equation*}\RC=\R\langle\underline{e}+\underline{\sigma}\rangle\oplus\R\langle\underline{e}-\underline{\sigma}\rangle.\end{equation*}This direct sum is $\C$-isomorphic to $\R\oplus \Rdelta$, where $\R$ and $\Rdelta$ are the trivial and sign $\C$-representations respectively with $\C$-actions defined below
\begin{align*}
    \sigma(x)&=x\quad (x\in\R),\\
    \sigma(y)&=-y\quad (y\in\Rdelta).
\end{align*}

Choosing the universe $\bigoplus\limits_{i=1}^{\infty} \RC$, we define a new indexing category as follows. 
\begin{definition}\index{$\CL$}\index{$\LL(U,V)$}
The \emph{equivariant indexing category} $\CL$ is a $\CTop_*$-enriched category whose objects are the finite dimensional subrepresentations of $\bigoplus\limits_{i=1}^{\infty} \RC$ with inner product. Let $\LL(U,V)$ denote the space of (not necessarily $\C$-equivariant) linear isometries with the $\C$-action that is conjugation. That is, for $f\in \LL(U,V)$, $e*f=f$ and $\sigma * f = \sigma f \sigma : U\rightarrow V$. The hom-object of morphisms $U\rightarrow V$ is the pointed $\C$-space $\LL(U,V)_+$. 
\end{definition}

From the discussion above on the decomposition of $\RC$, we can see that an object in the category $\CL$ is isomorphic to an object of the form $\R^{p+q\delta}=\R^p\oplus \R^{q\delta}$, for some $p,q\in \mathbb{N}$. That is, $p$ copies of the trivial representation and $q$ copies of the sign representation.  

\begin{rem}\index{$\R^{p,q}$}\index{$(p,q)V$}
We will use the notation $\R^{p,q}$ to mean $\R^{p+q\delta}$. We will also use the notation $(p,q)V$ to mean $\R^{p,q}\otimes V$ equipped with the diagonal action of $\C$, where $V\in\CL$.
\end{rem}

We can now define the input category for $\C$-orthogonal calculus, which we denote by $\Ezero$. 
\begin{definition}\label{jzero and ezero def}\index{$\Jzero$}
Define $\Jzero$ to be the $\CTop_*$-enriched category with the same objects as $\CL$ and morphisms defined by $\Jzero(U,V)=\LL(U,V)_+$. 
Define the \emph{input category} $\Ezero$ to be the category of $\CTop_*$-enriched functors from $\Jzero$ to $\CTop_*$ and $\C$-equivariant natural transformations (see Remark \ref{remNat} for details). 
\end{definition}

This category inherits a projective model structure from the fine model structure on $\C\TTop_*$. The projective model structure on $\Ezero$ is a special case of the projective model structure on $\OEpq$ for $p=q=0$, as such we defer the proof to Lemma \ref{objectwise model structure proof}.
\begin{prop}\label{proj model structure}\index{$\Ezero$}
There is a proper, cellular, $\C$-topological model structure on $\Ezero$ where the fibrations and weak equivalences are defined objectwise from the fine model structure on $\CTop_*$. We call them objectwise fibrations and objectwise weak equivalences. We call this the projective model structure on $\Ezero$ and denote it by $\Ezero$. It is cofibrantly generated by the following sets of generating cofibrations and generating acyclic cofibrations respectively 
\begin{align*}
    &I_{\proj}=\{\Jzero (V,-)\wedge i : i\in I_{\C}\}\\
    &J_{\proj}=\{\Jzero (V,-)\wedge j : j\in J_{\C}\},
\end{align*}
where $V\in \Jzero$ and $I_{\C},J_{\C}$ are the generating cofibrations and acyclic cofibrations of the fine model structure on $\CTop_*$. 
\end{prop}

\subsection{The intermediate categories}

In orthogonal calculus, one constructs intermediate categories $O(n)\mathcal{E}_n$ which are intermediate between the input category and the category of orthogonal spectra with an action of $O(n)$. We replicate this process in the $\C$-equivariant setting to construct new intermediate categories $\OEpq$. The construction follows the orthogonal calculus version of Weiss \cite[Sections 1 and 2]{Wei95} and Barnes and Oman \cite[Sections 3 and 8]{BO13}.

We begin by defining the following $\C$-equivariant vector bundle.
\begin{definition}\label{def: gamma pq bundle}\index{$\Gmor{p,q}{U}{V}$}
Let $U,V\in \CL$. Define the \emph{$(p,q)$-th complement bundle} $\Gmor{p,q}{U}{V}$ to be the $\C$-equivariant vector bundle on $\LL(U,V)$, whose total space is given by
\begin{equation*}
    \Gmor{p,q}{U}{V}=\{(f,x):f\in \LL(U,V), x\in \R^{p,q} \otimes f(U)^\perp\},
\end{equation*}
where $f(U)^\perp$ denotes the orthogonal complement of the image of $f$.

\noindent Let $(f,x)\in \Gmor{p,q}{U}{V} $. Define a $\C$-action on $\Gmor{p,q}{U}{V} $ by \begin{equation*}
\sigma(f,x)=(\sigma *f, \sigma x).
\end{equation*}
\end{definition}

The following result outlines the effect of the fixed point functor $(-)^{\C}:\CTop_*\rightarrow\TTop_*$ on the $C_2$-spaces $\LL(\mathbb{R}^{a,b},\mathbb{R}^{c,d})$ and $C_2 \gamma_{p,q}(\mathbb{R}^{a,b},\mathbb{R}^{c,d})$. 

\begin{thm}[The Equivariant Splitting Theorems]\label{splittingtheorems}There are homeomorphisms
\begin{equation*}
\LL(\mathbb{R}^{a,b},\mathbb{R}^{c,d})^{C_2}\cong \LL(\mathbb{R}^{a},\mathbb{R}^{c})\times \LL(\mathbb{R}^{b\delta},\mathbb{R}^{d\delta})
\end{equation*}
\begin{equation*}
C_2\gamma_{p,q}(\mathbb{R}^{a,b},\mathbb{R}^{c,d})^{C_2}\cong \C\gamma_{p,0}(\mathbb{R}^{a},\mathbb{R}^{c})\times \C\gamma_{0,q}(\mathbb{R}^{b\delta},\mathbb{R}^{d\delta}).
\end{equation*}
\end{thm}
\begin{proof}
The first result is an application of Schur's Lemma.

Let $(f,x)\in\Gmor{p,q}{\mathbb{R}^{a,b}}{\mathbb{R}^{c,d}}$. Then $(f,x)$ is $\C$-fixed if and only if $(\sigma *f,\sigma x)=(f,x)$. By the first splitting theorem, $f$ must be of the form $f_1\times f_2$, where $f_1\in\LL(\mathbb{R}^{a},\mathbb{R}^{c})$ and $f_2\in \LL(\mathbb{R}^{b\delta},\mathbb{R}^{d\delta})$. 

We know that $x\in \R^{p,q}\otimes \text{Im}(f)^\perp\subset \R^{p,q}\otimes \R^{c,d}=(\R^p\otimes\R^c)\oplus(\R^p\otimes\R^{d\delta})\oplus(\R^{q\delta}\otimes\R^{c})\oplus(\R^{q\delta}\otimes\R^{d\delta})$, where $\C$ acts as $\id\oplus -1\oplus -1\oplus\id$. Therefore $\sigma(x)=x$ if and only if $$x\in (\R^p\otimes\R^c)\oplus(\R^{q\delta}\otimes\R^{d\delta}).$$ That is, $x$ is of the form $x_1\oplus x_2$, where $x_1\in \R^p\otimes \text{Im}(f_1)^\perp$ and $x_2\in\R^{q\delta}\otimes\text{Im}(f_2)^\perp$. The second splitting theorem follows from the well defined homeomorphism 
\begin{equation*}
    (f_1\times f_2,x_1\oplus x_2)\mapsto((f_1,x_1),(f_2,x_2)).\qedhere
\end{equation*}\end{proof}

Now we define what will become the morphism spaces for the categories $\Jpq$. These categories are analogous to the $n$-th jet categories of orthogonal calculus, and will be used to build the intermediate categories. 

\begin{definition}\label{thom}\index{$\Jmor{p,q}{U}{V}$}
Let $U,V \in \CL$. Define $\Jmor{p,q}{U}{V}$ to be the Thom space of $\Gmor{p,q}{U}{V}$.
\end{definition}

This Thom space is the one point compactification of $\Gmor{p,q}{U}{V}$, since $\LL(U,V)$ is compact. Hence, each $\Jmor{p,q}{U}{V}$ is a pointed $\C$-space, with $\C$-action inherited from $\Gmor{p,q}{U}{V}$. 

We can now define the categories $\Jpq$. 
\begin{definition}\label{def: p,q-jet cat}\index{$\Jpq$}
Let the \emph{$(p,q)$-th jet category} $\Jpq$ be the $\CTop$-enriched category whose objects are the same as $\CL$, and whose morphisms are given by $\Jmor{p,q}{U}{V}$. 

Composition in $\Jpq$ is defined as follows. There are maps of $\C$-spaces defined by
\begin{align*}
\Gmor{p,q}{V}{X}\times \Gmor{p,q}{U}{V}&\rightarrow \Gmor{p,q}{U}{X}\\
((f,x),(g,y))&\mapsto (f\circ g,x+(\id\otimes f)(y)).
\end{align*}Passing to Thom spaces then yields the desired composition maps
\begin{equation*}
    \Jmor{p,q}{V}{X}\wedge \Jmor{p,q}{U}{V}\rightarrow \Jmor{p,q}{U}{X}.
 \end{equation*}
One can check that this composition is continuous, unital, associative, and $\C$-equivariant.
\end{definition}

One can see that the category $\Jzero$ has morphisms $\Jmor{0,0}{U}{V}=\LL(U,V)_+$, and therefore is exactly the category defined in Definition \ref{jzero and ezero def}.

The following is the a $\C$-equivariant generalisation of \cite[Theorem 1.2]{Wei95}. It demonstrates that it is possible to build the morphism spaces $\Jmor{p,q}{U}{V}$ inductively. That is, we can construct $C_2\mathcal{J}_{p+1,q}$ and $C_2\mathcal{J}_{p,q+1}$ from $\Jpq$. The proof is analogous to that of the underlying Theorem, and checking equivariance follows as in \cite[Proposition 2.6]{Tag22unit}.

\begin{prop}\label{cofibseq}
For all $U,V,W$ in $\Jzero$, the homotopy cofibre of the restricted composition map
\begin{equation*}
\Jmor{W}{U\oplus \mathbb{R}^\alpha}{V}\wedge S^{W\otimes \mathbb{R}^\alpha} \rightarrow \Jmor{W}{U}{V}
\end{equation*}
is $\C$-equivariantly homeomorphic to $\Jmor{W\oplus \mathbb{R}^\alpha}{U}{V}$, where $\mathbb{R}^\alpha$ is either the trivial or sign representation of $\C$, and $S^{W\otimes \mathbb{R}^\alpha}$ has been $\C$-equivariantly identified with the closure of the subspace of pairs $(i,x)$ in $\Jmor{W}{U}{U\oplus \mathbb{R}^\alpha}$ with $i$ the standard inclusion.
\end{prop}

\begin{rem}\label{rem: cofib seq}   
These cofibre sequences are analogous to the following cofibre sequences constructed in the underlying calculus (see \cite[Proposition 1.2]{Wei95}).  
\begin{equation*}
    \mathcal{J}_{n}(U\oplus \mathbb{R},V)\wedge S^{\mathbb{R}^n} \rightarrow \mathcal{J}_{n}(U,V)\rightarrow \mathcal{J}_{n+1} (U,V)
\end{equation*}

If one wanted to replace $\R$ with something higher dimensional this would involve ‘gluing’ these cofibre sequences together in an iterative manner. The $\C$-equivariant calculus has been constructed such that it reduces down to the underlying calculus after forgetting the $\C$-actions. This forces that the cofibre sequences in Proposition \ref{cofibseq} only hold when $\mathbb{R}^\alpha$ is either the trivial or sign $\C$-representation, since these cases both correspond $\R$ in the non-equivariant statement. To replace $\mathbb{R}^\alpha$ with something of higher dimension, for example $\R^{1+1\delta}$, would again mean taking some kind of iteration of cofibre sequences. This indicates that a potentially more involved approach may be needed if one wants to construct this kind of result in a $G$-equivariant orthogonal calculus, for an arbitrary group $G$. As a result, it is also not obvious how derivatives should behave for the arbitrary $G$ setting, since the fibre sequences that describe derivatives (see Proposition \ref{loops fibre sequence}) are a direct consequence of these cofibre sequences. 
\end{rem}

We now define the functor categories $\Epq$. At the same time, we will also define functor categories $\OEpq$, which will later be used to classify the layers of the orthogonal tower. Before we can do this, we introduce the group $O(p,q)$ and discuss its actions. 
\begin{definition}\label{def:O(p,q) and matrix A}\index{$O(p,q)$}\index{$O(p,q)\rtimes\C$}
Define $O(p,q)$ to be the group of linear isometries from $\mathbb{R}^{p,q}$ to $\mathbb{R}^{p,q}$ with the conjugation $\C$-action. In particular, $O(p,q)^{\C}=O(p)\times O(q)$. 
\end{definition}

There is an action of $O(p,q)\rtimes \C$ on $\mathbb{R}^{p,q}$ given by $(T,\sigma)(x):= T(\sigma (x))$. This can be extended to an action on $\mathbb{R}^{p,q} \otimes f(U)^\perp$ by $(T,\sigma)x:=((T,\sigma) \otimes \sigma) (x)$, where $f\in \LL(U,V)$. This induces an $O(p,q)\rtimes \C$-action on $\Gmor{p,q}{U}{V}$ by $$(T,\sigma)(f,x):=(\sigma *f,((T,\sigma) \otimes\sigma)(x)).$$ Hence, there is also an action of $O(p,q)\rtimes \C$ on its Thom space $\Jmor{p,q}{U}{V}$, making $\Jpq$ an $\COTop_* $-enriched category. 

\begin{rem}\label{rem: what is the ms on semi direct prod}
Note that the actions of $\C$ and $O(p,q)$ do not commute. Throughout this paper, we equip the category of $\C$-spaces with the fine model structure and the category of $O(p,q)$-spaces with the coarse model structure. We then equip the category of $(O(p,q)\rtimes \C)$-spaces with a model structure which is fine with respect to $\C$ and coarse with respect to $O(p,q)$. This is discussed more in Remark \ref{remark: model structure on semi direct}. 
\end{rem}

\begin{prop}\label{sphere as quotient of orthogonal groups prop}
For all $p>0$ and $q\geq 0$, there exists a $\C$-equivariant homeomorphism 
\begin{equation*}
    O(p,q)/O(p-1,q)\cong S(\mathbb{R}^{p+q\delta}).
\end{equation*}
\end{prop}
\begin{proof}

Since $O(p,q)$ acts on $\mathbb{R}^{p+q\delta}$ transitively by linear isometries, there is a restricted transitive action of $O(p,q)$ on $S(\mathbb{R}^{p+q\delta})$. Fix the vector $e_1=(1,0,...,0)$ in $S(\mathbb{R}^{p+q\delta})$. There is a continuous $\C$-equivariant map $\phi: O(p,q)\rightarrow S(\mathbb{R}^{p+q\delta})$ given by $g\mapsto ge_1$. The stabiliser of $e_1$ is the subgroup of $O(p,q)$ given by 
\begin{equation*}
    \biggl\{ \begin{pmatrix} 1 & 0\\0&A\end{pmatrix}:A\in O(p-1,q) \biggr\}, 
\end{equation*}
which is $\C$-homeomorphic to $O(p-1,q)$. 
The quotient $O(p,q)/O(p-1,q)$ inherits a $\C$-action defined by $\sigma ([g]):= [\sigma (g)]$. It follows by the orbit-stabiliser theorem that there is a continuous $\C$-equivariant homeomorphism $O(p,q)/O(p-1,q)\cong S(\mathbb{R}^{p+q\delta})$, which is given by $[g]\mapsto ge_1$. 
\end{proof}

\begin{definition}\index{$\Epq$}\index{$\OEpq$}
Define $\Epq$ to be the category of $\CTop_*$ enriched functors from the $(p,q)$-th jet category $\Jpq$ to $\CTop_*$ and $\C$-equivariant natural transformations, denoted by $\C \Nat_{p,q}(-,-)$.

Define the \emph{$(p,q)$-th intermediate category} $\OEpq$ to be the category of $\COTop_*$-enriched functors from the $(p,q)$-th jet category $\Jpq$ to $\COTop_*$, and $(O(p,q)\rtimes\C)$-equivariant natural transformations.
\end{definition}

For $p,q=0$ this definition is exactly the category $\Ezero$ in Definition \ref{jzero and ezero def}.

\begin{rem}\label{remNat}\index{$\Nat_{p,q}(E,F)$}\index{$\C\Nat_{p,q}(E,F)$}

The set of natural transformations between $E,F\in \Epq$ is denoted by $\Nat_{p,q}(E,F)$. There is a natural topology on $\Nat_{p,q}(E,F)$, which is the subspace topology of a product space as follows. 
\begin{align*}
    \Nat_{p,q} (E,F)&:= \int\limits_{V\in\Jpq} \TTop_* (E(V),F(V))\\
    &\subseteq \prod\limits_{V\in\Jpq} \TTop_* (E(V),F(V))
\end{align*}
There is a $\C$-action on the space of natural transformations $\Nat_{p,q}(E,F)$ induced by the conjugation action on $\TTop_* (E(V),F(V))$. This defines an enrichment of $\Epq$ in $\CTop_*$. 

With respect to this conjugation action, we topologise the set of $\C$-equivariant natural transformations between $E,F\in\Epq$, denoted by $\C\Nat_{p,q}(E,F):=\Nat_{p,q}(E,F)^{\C}$ as follows. 
\begin{align*}
    \C\Nat_{p,q} (E,F)&:= \int\limits_{V\in\Jpq} \C\TTop_* (E(V),F(V))\\
    &\subseteq \prod\limits_{V\in\Jpq} \C\TTop_* (E(V),F(V))
\end{align*}

Similar descriptions exist for the morphisms in $O(p,q)\Epq$. 

We can describe a functor $E\in \Epq$ in terms of an enriched coend (and similarly for $O(p,q)\Epq$), by the Yoneda lemma (see for example \cite[Section 3.10]{Kel05}). 
\begin{equation*}
    \int^{W\in\Jpq} E(W) \wedge \Jpq(W,-)\cong E.
\end{equation*}
Alternatively, we can describe a functor $E\in \Epq$ in terms of natural transformations, by the enriched Yoneda lemma.
\begin{equation*}
    E(W)\cong\Nat_{p,q}(\Jpq(W,-),E)=\int\limits_{V\in\Jpq} \TTop_* (\Jpq(W,V),E(V)).
\end{equation*}

Another useful result, that we use throughout the paper, is that the functor $\Nat_{p,q}(-,F)$ sends homotopy cofibre sequences to homotopy fibre sequences. This follows from the fact that the functor $\TTop_*(-,A):\C\TTop_*\rightarrow \C\TTop_*$ sends homotopy cofibre sequences to homotopy fibre sequences (it is contravariant, sends colimits to limits, and $\CTop_*$ is closed symmetric monoidal) and using the definition of $\Nat_{p,q}(-,F)$ as the end above. 

\end{rem}

\subsection{Derivatives}

Derivatives play a key role in calculus of real functions. They describe the difference between successive polynomial approximations in the Taylor series. As the name calculus suggests, one can define a notion of derivatives of functors in orthogonal calculus, as done by Weiss in \cite[Section 2]{Wei95} and Barnes and Oman in \cite[Section 4]{BO13}. In this section, we will extend this theory to the $\C$-equivariant setting. Derivatives play a key role in the classification of $(p,q)$-homogeneous functors, as the derivative adjunctions form one half of the zig-zag of equivalences between the $(p,q)$-homogeneous model structure and the category of orthogonal $\C$-spectra with an action of $O(p,q)$, see Theorem \ref{zigzagclassification}. 

Let $i_{p,q}^{l,m}:\R^{p,q}\rightarrow\R^{l,m}$\index{$i_{p,q}^{l,m}$} be the $\C$-equivariant inclusion map $(x,y)\mapsto (x,0,y,0)$, where $p\leq l$ and $q\leq m$. Such a map induces a group homomorphism $O(p,q)\rightarrow O(l,m)$, which is $O(p,q)$-equivariant by letting $O(p,q)$ act on the first $p$ and $q$ coordinates of $O(l,m)$. That is, both $\R^{p,q}$ and $\R^{l,m}$ are $(O(p,q)\rtimes \C)$-spaces. 

This map induces a map of $(O(p,q)\rtimes\C)$-equivariant spaces 
\begin{align*}
(i_{p,q}^{l,m} )_{U,V}:\Gmor{p,q}{U}{V}&\rightarrow\Gmor{l,m}{U}{V}\\
(f,x)&\mapsto(f,(i_{p,q}^{l,m}\otimes id)(x))   
\end{align*}
which in turn induces a map on the $(O(p,q)\rtimes\C)\TTop_*$-enriched categories $\CJ{p}{q}\rightarrow\CJ{l}{m}$. 

\begin{definition}\index{$\res_{p,q}^{l,m}$}\index{$\res_{p,q}^{l,m}/O(l-p,m-q)$}
Let $p\leq l$ and $q\leq m$.

\noindent Define the \emph{restriction functor} $\res_{p,q}^{l,m}:\CE{l}{m} \rightarrow \CE{p}{q}$ as precomposition with $i_{p,q}^{l,m}$.

\noindent Define the \emph{restriction-orbit functor} by
\begin{align*}
\res_{p,q}^{l,m}/O(l-p,m-q):\OCE{l}{m}&\rightarrow\OCE{p}{q}\\
F&\mapsto (F\circ i_{p,q}^{l,m})/O(l-p,m-q).
\end{align*}

\end{definition}

The restriction and restriction-orbit functors have right adjoints. Before we can define them, we must define an adjoint to the orbit functor, see \cite[Lemma 4.2]{BO13}. 

\begin{lem}
Let $p\leq l$ and $q\leq m$. There is an adjoint pair 
\begin{equation*}
    (-)/O(l-p,m-q):(O(l,m)\rtimes\C)\TTop_* \rightleftarrows \COTop_* : \CI^{l,m}_{p,q}.
\end{equation*}
An $(O(p,q)\rtimes\C)$-space $A$ can be considered as an $((O(p,q)\times O(l-p,m-q))\rtimes\C)$-space, by letting $O(l-p,m-q)$ act trivially. Call this space $\varepsilon^*A$. Define $\CI^{l,m}_{p,q}A$ to be the space of $(O(p,q)\times O(l-p,m-q))$-equivariant maps from $O(l,m)\rtimes \C$ to $\varepsilon^*A$, which has the $(O(l,m)\rtimes\C)$-action induced by the conjugation $\C$-action and the action of $O(l,m)$ on itself. 
\end{lem}

\begin{definition}\label{induction definition}\index{$\ind_{p,q}^{l,m}$}\index{$\ind_{p,q}^{l,m}\CI$}\index{$\ind_{0,0}^{l,m}\varepsilon^*$}
Let $p\leq l$ and $q\leq m$.

\noindent Define the \emph{induction functor} $\ind_{p,q}^{l,m}:\CE{p}{q} \rightarrow \CE{l}{m}$ by
\begin{equation*}
    \ind_{p,q}^{l,m}F:U\mapsto\Nat_{p,q}(\CJ{l}{m}(U,-),F),
\end{equation*}
where the space of natural transformations of objects of $\Epq$ is equipped with the conjugation $\C$-action (see Remark \ref{remNat}).
\noindent 

Define the \emph{inflation-induction functor} $\ind_{p,q}^{l,m}\CI:\OCE{p}{q} \rightarrow \OCE{l}{m}$ by
\begin{equation*}
    \ind_{p,q}^{l,m}\CI F:U\mapsto\Nat_{\OEpq}(\CJ{l}{m}(U,-),\CI_{p,q}^{l,m}\circ F).
\end{equation*}
When $p,q=0$, $\CI_{p,q}^{l,m}$ simply gives $F$ the trivial $O(l,m)$-action, hence we write $\ind_{0,0}^{l,m}\CI F$ as $\ind_{0,0}^{l,m}\varepsilon^*F$. This is what we call the $(l,m)$-th derivative of $F$, denoted by $$F^{(l,m)}:=\ind_{0,0}^{l,m}\varepsilon^* F.$$\index{$(-)^{(p,q)}$} 
\end{definition}

\begin{lem}
The induction functor $\ind_{p,q}^{l,m}$ is right adjoint to the restriction functor $\res_{p,q}^{l,m}$.
The inflation-induction functor $\ind_{p,q}^{l,m}\CI$ is right adjoint to the restriction-orbit functor $\res_{p,q}^{l,m}/O(l-p,m-q)$.
\end{lem}

\begin{rem}
There are commutative diagrams of categories
% https://q.uiver.app/#q=WzAsNCxbMCwwLCJcXENcXG1hdGhjYWx7RX1fe3AscX0iXSxbMiwwLCJcXENcXG1hdGhjYWx7RX1fe3ArMSxxfSJdLFswLDIsIlxcQ1xcbWF0aGNhbHtFfV97cCxxKzF9Il0sWzIsMiwiXFxDXFxtYXRoY2Fse0V9X3twKzEscSsxfSJdLFswLDEsIlxcaW5kX3twLHF9XntwKzEscX0iXSxbMSwzLCJcXGluZF97cCsxLHF9XntwKzEscSsxfSJdLFswLDIsIlxcaW5kX3twLHF9XntwLHErMX0iLDJdLFsyLDMsIlxcaW5kX3twLHErMX1ee3ArMSxxKzF9IiwyXSxbMCwzLCJcXGluZF97cCxxfV57cCsxLHErMX0iLDJdXQ==
\[\begin{tikzcd}
	{\C\mathcal{E}_{p,q}} && {\C\mathcal{E}_{p+1,q}} \\
	\\
	{\C\mathcal{E}_{p,q+1}} && {\C\mathcal{E}_{p+1,q+1}}
	\arrow["{\ind_{p,q}^{p+1,q}}", from=1-1, to=1-3]
	\arrow["{\ind_{p+1,q}^{p+1,q+1}}", from=1-3, to=3-3]
	\arrow["{\ind_{p,q}^{p,q+1}}"', from=1-1, to=3-1]
	\arrow["{\ind_{p,q+1}^{p+1,q+1}}"', from=3-1, to=3-3]
	\arrow["{\ind_{p,q}^{p+1,q+1}}"', from=1-1, to=3-3]
\end{tikzcd}\]

As in \cite{Wei95}, the induction functors give us a notion of differentiation of functors in our input category $\Ezero$ (see Definition \ref{jzero and ezero def}). In particular, for the $\C$-equivariant case, there are two directions in one can take a derivative; in the $p$ direction and in the $q$ direction. These two different directions of differentiating can be thought of as partial derivatives, and then the commuting diagram above tells us that taking both possible orders of mixed partial derivatives is the same as taking the total derivative. 
\end{rem}

The following proposition defines induction iteratively as a homotopy fibre, and acts as a tool for calculating the derivatives. 

\begin{prop}\label{loops fibre sequence}
For all $U\in\Jzero$ and for all $F\in\Epq$, there are homotopy fibre sequences of $\C$-spaces
\begin{equation*}
\res_{p,q}^{p+1,q}\ind_{p,q}^{p+1,q} F(U)\rightarrow F(U)\rightarrow \Omega^{(p,q)\mathbb{R}}F(U\oplus \R)
\end{equation*}
and 
\begin{equation*}
\res_{p,q}^{p,q+1}\ind_{p,q}^{p,q+1} F(U)\rightarrow F(U)\rightarrow \Omega^{(p,q)\Rdelta}F(U\oplus \Rdelta),
\end{equation*}
where $\Omega^{(p,q)V}Y$ represents the space of maps $S^{(p,q)V}\rightarrow Y$, for a $\C$-space $Y$, and is given the conjugation $\C$-action. 
\end{prop}

\begin{proof}
This follows in the same way as for the underlying version \cite[Theorem 2.2]{Wei95}, using the equivariant version of \cite[Theorem 1.2]{Wei95}, Proposition \ref{cofibseq}.
\end{proof}

\subsection{The $(p,q)$-stable model structure}

We want to compare the $(p,q)$-th intermediate category $\OEpq$ with the category of orthogonal $\C$-spectra with an action of $O(p,q)$. The $(p,q)$-stable model structure constructed will be a modification of the stable model structure on orthogonal $\C$-spectra, see \cite[Section 3.4]{MM02}. This modification will account for the fact that the structure maps of objects in $\OEpq$ are of the form 
\begin{equation*}
    \sigma_X:S^{(p,q)V}\wedge X(W)\rightarrow X(W\oplus V).
\end{equation*}

We begin by defining two functors. These functors form an adjunction between the categories $\COTop_*$ and $\CTop_*$.  

\begin{definition}
Let $\beta:\C\rightarrow O(p,q)\rtimes \C$ be defined by $\alpha\mapsto (\Id_{p+q},\alpha)$.

Let $\beta^*$ be the restriction functor $\COTop_*\rightarrow \CTop_*$, which sends $X$ to the underlying space $X$ with $\C$-action given by $\sigma x= (\beta (\sigma)) x$. 

Let $\beta_!$ be the functor $\CTop_*\rightarrow \COTop_*$ defined by 
\begin{equation*}
    X\mapsto  (O(p,q)\rtimes \C)_+\wedge_{\C} X.
\end{equation*}
\end{definition}

There is a projective model structure on the intermediate categories $\OEpq$, similar to the levelwise model structure constructed by Barnes and Oman \cite[Lemma 7.6]{BO13}, in which fibrations and weak equivalences are defined objectwise. A left Bousfield localisation of this model structure will give the $(p,q)$-stable model structure. This projective model structure is exactly the level model structure of \cite[Section 6]{MMSS01}. 

\begin{definition}\label{def: obj WE}
Let $f:X\rightarrow Y$ be a map in $\OEpq$. Call $f$ an \emph{objectwise fibration} or an \emph{objectwise weak equivalence} if $\beta^*(f(U)):\beta^*(X(U))\rightarrow \beta^*(Y(U))$ is a fibration or weak equivalence of pointed $\C$-spaces, for each $U\in \Jzero$. Call $f$ a \emph{cofibration} if it has the left lifting property with respect to the objectwise acyclic fibrations. 
\end{definition}

\begin{lem}\label{objectwise model structure proof}\index{$\OEpql$}
There is a cellular, proper, $\C$-topological model structure on the $(p,q)$-th intermediate category $\OEpq$ formed by the objectwise weak equivalences and objectwise fibrations. Denote this model category by $\OEpql$ and call it the projective model structure on $\OEpq$. The generating cofibrations and generating acyclic cofibrations are given by $\Ilevel$ and $\Jlevel$ respectively,
\begin{align*}
    &\Ilevel=\{\Jmor{p,q}{U}{-}\wedge  \beta_!(i) : U\in\Jzero, i\in I_{\C}\}\\
    &\Jlevel=\{\Jmor{p,q}{U}{-}\wedge \beta_!(j) : U\in\Jzero, j\in J_{\C}\}
\end{align*}
where $I_{\C},J_{\C}$ are the generating cofibrations and acyclic cofibrations of the fine model structure on $\CTop_*$.
\end{lem}

\begin{rem}\label{remark: model structure on semi direct}
Similarly, there exists a cellular, proper, $\C$-topological model structure on the category $(O(p,q)\rtimes \C)\TTop_*$, where weak equivalences and fibrations are defined by restricting to $\CTop_*$ along $\beta^*$. The generating (acyclic) cofibrations are of the form $\beta_! (i)$ where $i$ is a generation (acyclic) cofibration of $\CTop_*$. This is exactly the model structure which is coarse with respect to $O(p,q)$ and fine with respect to $\C$ (see Remark \ref{rem: what is the ms on semi direct prod}).
\end{rem}

We will use this projective model structure to construct the $(p,q)$-stable model structure using a left Bousfield localisation. In the same was as for $\C$-spectra \cite[Chapter 3]{MM02}, we begin by first defining homotopy groups on objects of $\OEpq$.

\begin{definition}\label{def: pq pi equiv}\index{$(p,q)\pi_k^H (-)$}
Define the \emph{$(p,q)$-homotopy groups} of $X\in\OEpq$ by 
\begin{equation*}
(p,q)\pi_k^H X =
\left\{
	\begin{array}{ll}
		\colim_V \pi_k\left( \Omega^{(p,q)V} X(V)\right)^H,  & \mbox{if } k \geq 0 \\
		\colim_{V\supset \mathbb{R}^{|k|}} \pi_0\left( \Omega^{p,q({V-\mathbb{R}^{|k|}})} X(V)\right)^H,& \mbox{if } k < 0
	\end{array}
\right.
\end{equation*} 
where $V$ runs over the indexing $\C$-representations in $\CL$, $H\leq\C$ is a closed subgroup, and $V-\mathbb{R}^{|k|}$ denotes the orthogonal complement of $\mathbb{R}^{|k|}$ in $V$. Define a map $f:X\rightarrow Y$ in $\OEpq$ to be a $(p,q)\pi_*$-equivalence if the induced map $(p,q)\pi_k^H f: (p,q)\pi_k^H X\rightarrow (p,q)\pi_k^H Y$ is an isomorphism for all $k$ and all closed subgroups $H\leq \C$. 
\end{definition}

One can easily verify that if $\C$ was replaced by the trivial group, the $(p,q)$-homotopy groups for $p+q=n$ are exactly the $n$-homotopy groups defined in \cite[Definition 7.7]{BO13}. 

Now we identify the fibrant objects of the $(p,q)$-stable model structure. These are a generalisation of $\Omega$-spectra, which are the fibrant objects of the stable model structure on orthogonal spectra (see Barnes and Roitzheim \cite[Corollary 5.2.17]{BR20}). 

\begin{definition}\label{def: pq omega spectrum}
An object $X$ of $\OEpq$ has structure maps $$\sigma_X:S^{(p,q)V}\wedge X(W)\rightarrow X(W\oplus V)$$induced by the identification of $S^{(p,q)V}$ as a subspace of $\Jpq(W,W\oplus V)$ and the structure maps of $X$ being an enriched functor. The object $X$ is called a \emph{$(p,q)\Omega$-spectrum} if its adjoint structure maps $$\tilde{\sigma}_X:X(W)\rightarrow \Omega^{(p,q)V} X(W\oplus V)$$\index{$\tilde{\sigma}_X$}are weak equivalences of $\C$-spaces, for all $V,W\in \Jzero$. \end{definition}

\begin{lem}\label{omega spectra lemma}
$X$ is a $(p,q)\Omega$-spectrum if and only if for all $W\in\Jzero$ the maps
\begin{align*}
    X(W)&\rightarrow\Omega^{(p,q)\R}X(W\oplus\R)\\
    X(W)&\rightarrow\Omega^{(p,q)\Rdelta}X(W\oplus\Rdelta)
\end{align*}
are weak equivalences of $\C$-spaces.
\end{lem}
\begin{proof}
If $X$ is a $(p,q)\Omega$-spectrum, then clearly both maps are weak equivalences by setting $V=\R$ and $V=\Rdelta$ in Definition \ref{def: pq omega spectrum}. 

If $X$ is such that the two maps are weak equivalences, then $X$ being a $(p,q)\Omega$-spectrum follows by repeated application of the weak equivalences, as demonstrated in the diagram below. 
% https://q.uiver.app/?q=WzAsMTAsWzAsMCwiWChWKSJdLFsxLDAsIlxcT21lZ2Fee3AscVxcbWF0aGJie1J9fVgoVlxcb3BsdXNcXG1hdGhiYntSfSkiXSxbMywwLCJcXE9tZWdhXntwLHFcXG1hdGhiYntSfV5tfVgoVlxcb3BsdXNcXG1hdGhiYntSfV5tKSJdLFswLDEsIlxcT21lZ2Fee3AscVxcbWF0aGJie1J9XlxcZGVsdGF9WChWXFxvcGx1c1xcbWF0aGJie1J9XlxcZGVsdGEpIl0sWzAsMywiXFxPbWVnYV57cCxxXFxtYXRoYmJ7Un1ee25cXGRlbHRhfX1YKFZcXG9wbHVzXFxtYXRoYmJ7Un1ee25cXGRlbHRhfSkiXSxbMSwxLCJcXE9tZWdhXntwLHFcXG1hdGhiYntSfV57MSwxfX1YKFZcXG9wbHVzXFxtYXRoYmJ7Un1eezEsMX0pIl0sWzMsMywiXFxPbWVnYV57cCxxXFxtYXRoYmJ7Un1ee20sbn19WChWXFxvcGx1c1xcbWF0aGJie1J9XnttLG59KSJdLFswLDIsIlxcdmRvdHMiXSxbMiwyLCJcXGRkb3RzIl0sWzIsMCwiXFxkb3RzIl0sWzAsMSwiXFxzaW1lcSJdLFswLDMsIlxcc2ltZXEiLDJdLFszLDUsIlxcc2ltZXEiLDJdLFs3LDQsIlxcc2ltZXEiLDJdLFsxLDUsIlxcc2ltZXEiXSxbMiw2LCJcXHNpbWVxIiwwLHsic3R5bGUiOnsiYm9keSI6eyJuYW1lIjoiZG90dGVkIn19fV0sWzQsNiwiXFxzaW1lcSIsMix7InN0eWxlIjp7ImJvZHkiOnsibmFtZSI6ImRvdHRlZCJ9fX1dLFs5LDIsIlxcc2ltZXEiXSxbMSw5LCJcXHNpbWVxIl0sWzMsNywiXFxzaW1lcSIsMl1d
\[\begin{tikzcd}
	{X(W)} & {\Omega^{(p,q)\mathbb{R}}X(W\oplus\mathbb{R})} & \dots & {\Omega^{(p,q)\mathbb{R}^m}X(W\oplus\mathbb{R}^m)} \\
	{\Omega^{(p,q)\mathbb{R}^\delta}X(W\oplus\mathbb{R}^\delta)} & {\Omega^{(p,q)\mathbb{R}^{1,1}}X(W\oplus\mathbb{R}^{1,1})} \\
	\vdots && \ddots \\
	{\Omega^{(p,q)\mathbb{R}^{n\delta}}X(W\oplus\mathbb{R}^{n\delta})} &&& {\Omega^{(p,q)\mathbb{R}^{m,n}}X(W\oplus\mathbb{R}^{m,n})}
	\arrow["\simeq", from=1-1, to=1-2]
	\arrow["\simeq"', from=1-1, to=2-1]
	\arrow["\simeq"', from=2-1, to=2-2]
	\arrow["\simeq"', from=3-1, to=4-1]
	\arrow["\simeq", from=1-2, to=2-2]
	\arrow["\simeq", dotted, from=1-4, to=4-4]
	\arrow["\simeq"', dotted, from=4-1, to=4-4]
	\arrow["\simeq", from=1-3, to=1-4]
	\arrow["\simeq", from=1-2, to=1-3]
	\arrow["\simeq"', from=2-1, to=3-1]
\end{tikzcd}\]
\end{proof}

We now want to identify the class of maps that will be used in the left Bousfield localisation. Let $\lambda_{V,W}^{p,q} :\Jmor{p,q}{W\oplus V}{-}\wedge S^{(p,q)W} \rightarrow \Jmor{p,q}{V}{-}$ be the restricted composition map, where $S^{(p,q)W}$ has been $\C$-equivariantly identified with the closure of the subspace of pairs $(i,x)$ in $\Jmor{p,q}{V}{W\oplus V}$ with $i$ the standard inclusion.

The maps $\lambda_{V,W}^{p,q}$ are $(p,q)\pi_*$-equivalences, by a similar argument as in the non-equivariant case \cite[Lemma 7.12]{BO13}. Now we will follow the same procedure as Mandell and May \cite[Section 3.4]{MM02} to turn these maps into cofibrations, in order to make generating sets for the $(p,q)$-stable model structure. 

Let $M\lambda_{V,W}^{p,q}$ be the mapping cylinder of $\lambda_{V,W}^{p,q}$. Then the map $\lambda_{V,W}^{p,q}$ can be factored as a cofibration $k_{V,W}^{p,q}$ and a deformation retract $r_{V,W}^{p,q}$ as follows. 
\begin{equation*}
    \Jmor{p,q}{W\oplus V}{-}\wedge S^{p,q W}\xrightarrow{k_{V,W}^{p,q}} M\lambda_{V,W}^{p,q} \xrightarrow{r_{V,W}^{p,q}} \Jmor{p,q}{V}{-}
\end{equation*}

\begin{definition}\index{$\Jstable$}
Define $\Jstable=\Jlevel\cup \{i\square k_{V,W}^{p,q} : i\in I_{\C}\text{ and } V,W\in \Jzero\}$, where $f\square g $ denotes the pushout product of two maps $f:A\rightarrow B$ and $g:X\rightarrow Y$, which is defined by 
\begin{equation*}
    f\square g: A\wedge Y \coprod_{A\wedge X} B\wedge X \rightarrow B\wedge Y.
\end{equation*}
\end{definition}

\begin{prop}\label{prop: stable ms}\index{$\OEpq^s$}
There is a cofibrantly generated, proper, cellular $\C$-topological model structure on the $(p,q)$-th intermediate category $\OEpq$ called the $(p,q)$-stable model structure. The cofibrations are the same as for the projective model structure, the weak equivalences are the $(p,q)\pi_*$-equivalences, and the fibrant objects are the $(p,q)\Omega$-spectra. The generating cofibrations and generating acyclic cofibrations are the sets $\Ilevel$ and $\Jstable$ respectively. Denote this model category by $\OEpq^s$.
\end{prop}

\begin{proof}
The $(p,q)$-stable model structure on $\OEpq$ is exactly the left Bousfield localisation of the projective model structure with respect to the class of maps $\lambda_{V,W}^{p,q}$.
\end{proof}

\section{Equivariant Polynomial functors}

In differential calculus, polynomial functions and derivatives are used to approximate real functions via the Taylor series. In orthogonal calculus, Weiss defines a class of input functors with properties analogous to those of polynomial functions, called polynomial functors, see \cite[Section 5]{Wei95}. These polynomial functors can be constructed into a tower that approximates a given functor, much like the Taylor series does for functions. It is the fibres of the maps between these polynomial approximation functors that are classified as spectra by the classification theorem, see \cite[Theorem 9.1]{Wei95} and \cite[Theorem 10.3]{BO13}. 

One can define a class of functors with polynomial properties in the $\C$-equivariant input category $\Ezero$. The addition of the $\C$-action makes it necessary to introduce an indexing shift from the underlying calculus. In particular, $\tau_n$ in the underlying calculus is defined using the poset $\{0\neq U\subseteq \R^{n+1}\}$ and $\taupq$ in the $\C$-calculus is defined using the poset $\{0\neq U\subseteq \R^{p,q}\}$.

\subsection{Polynomial functors}
In this section, we will adapt the definition of polynomial functors from the underlying calculus to fit the new $\C$-equivariant categories defined in Section \ref{sec: equiv functor cats}. We begin by defining a functor $\taupq$. The $(p,q)$-th complement bundle $\Gmor{p,q}{U}{V}$ has an associated sphere bundle $S\Gmor{p,q}{U}{V}$\index{$S\Gmor{p,q}{U}{V}$}. Considering $S\Gmor{p,q}{-}{-}:\Jzero^{op}\times \Jzero\rightarrow \CTop_*$ as a $\C\TTop_*$-enriched functor, we define the functor $\taupq:\Ezero\rightarrow \Ezero$ as follows.
\begin{definition}\index{$\taupq$}
Let $E\in \Ezero$. Define the functor $\taupq E\in\Ezero$ by 
\begin{equation*}
    \taupq E(V)=\Nat_{0,0}(S\Gmor{p,q}{V}{-}_+, E).
\end{equation*}
\end{definition}

There is an alternative description of $\taupq$ as a homotopy limit, that can be derived as a consequence of the following Proposition, which is a $\C$-equivariant generalisation of \cite[Proposition 4.2]{Wei95}. The proposition states that the sphere bundle of the vector bundle $\Gmor{p,q}{U}{V}$ can be written as a homotopy colimit, and in particular we can do this in a $\C$-equivariant way. The proof is analogous to that of the underlying Proposition, and checking that it holds equivariantly follows as in \cite[Proposition 2.9]{Tag22real}.

Recall that $\Jzero$ is the category of finite dimensional subrepresentations of $\bigoplus\limits_{i=1}^{\infty} \RC$ with inner product and with morphism spaces $\Jzero(U,V)=\LL (U,V)_+$ (Definition \ref{jzero and ezero def}). 
\begin{definition}\label{def: jzerobar}
Let $\Jzerobar$\index{$\Jzerobar$} denote the $\TTop_*$-enriched category of finite dimensional subspaces of $\bigoplus\limits_{i=1}^{\infty} \RC$ with inner product and with morphism spaces $\Jzerobar(U,V)=\LL (U,V)_+$.
\end{definition} 
There is a $\TTop_*$-enriched inclusion functor $\Jzero \hookrightarrow \Jzerobar$, since $\Jzero$ forms a subcategory of $\Jzerobar$. 

\begin{prop}\label{hocolimprop}
For all $V,W,X\in \Jzero$, there is a $\C$-homeomorphism 
\begin{equation*}
    S\Gmor{X}{V}{W}_+\cong \underset{0\neq U \subseteq X}
    {\hocolim} \Jzerobar (U\oplus V,W),
\end{equation*}
where $U$ is a non-zero subspace of $X$. 
\end{prop}

Note that the categories $\Jzero$ and $\CTop_*$ are $\TTop_*$-enriched, and that $\C\TTop_*$ is powered over $\TTop_*$. We use the notation $\overline{E}$\index{$\overline{E}$} to represent the right Kan extension of an input functor $E\in\Ezero$ along the inclusion $\Jzero\hookrightarrow\Jzerobar$. In particular, $\overline{E}$ is the $\TTop_*$-enriched functor $\Jzerobar \rightarrow \C\TTop_*$ defined by 
\begin{equation*}
    \overline{E}(X)=\int\limits_{W\in\Jzero} \TTop_*\left(\Jzerobar (X,W),E(W)\right),
\end{equation*}
with the $\C$-action induced by the following $\C$-action on $\TTop_*\left(\Jzerobar (X,W),E(W)\right)$
\begin{equation*}
    \sigma * f (x) := \sigma (f(x))
\end{equation*}
for all $f\in \TTop_*\left(\Jzerobar (X,W),E(W)\right)$ and $x\in \Jzerobar (X,W)$. For details of this construction see \cite[Chapter 4]{Kel05}.

\begin{prop}
Let $E\in\Ezero$. There is a $\C$-equivariant homeomorphism $$\taupq E(V) \cong \underset{0\neq U \subseteq \R^{p,q}}{\holim} \overline{E}(U\oplus V).$$
\begin{proof}
Using Proposition \ref{hocolimprop}, we get that
\begin{align*}
\taupq E(V)&=\Nat_{0,0}(S \Gmor{p,q}{V}{-}_+,E) \\
&\cong \Nat_{0,0} \left( \underset{0\neq U \subseteq \R^{p,q} } {\hocolim} \Jzerobar (U\oplus V,-),E\right)\\
&= \int\limits_{W\in\Jzero} \TTop_*\left(\underset{0\neq U \subseteq \R^{p,q} } {\hocolim} \Jzerobar (U\oplus V,W),E(W)\right)\\
&\cong \underset{0\neq U \subseteq \R^{p,q}}{\holim} \int\limits_{W\in\Jzero} \TTop_*\left(\Jzerobar (U\oplus V,W),E(W)\right)\\
&= \underset{0\neq U \subseteq \R^{p,q}}{\holim} \overline{E}(U\oplus V)
\end{align*}    
The homeomorphism in line 4 above is the comparison, made by Weiss, between the homotopy colimit and homotopy limit constructions (see \cite[Proposition 5.2]{Wei95}). The comparison used in the underlying calculus is naturally $\C$-equivariant.
\end{proof}
\end{prop}
This homotopy limit has a $\C$-action, since it can be expressed as the totalization of a cosimplicial space, which has a $\C$-action. This is discussed in more detail in Lemma \ref{e.3}.

Now we define what it means for a functor to be (strongly) polynomial. This definition is a $\C$-equivariant version of \cite[Definition 5.1]{Wei95}. 

\begin{definition}\label{def: polynomial}
A functor $E\in\Ezero$ is called \emph{strongly $(p,q)$-polynomial} if and only if the map 
\begin{equation*}
    \rho_{p,q}E:E\rightarrow\taupq E
\end{equation*}
is an objectwise weak equivalence. 

A functor $E\in\Ezero$ is called \emph{$(p,q)$-polynomial} if and only if $E$ is both strongly $(p+1,q)$-polynomial and strongly $(p,q+1)$-polynomial. 
\end{definition}

The fibre of the map $\rho_{p,q}E$ determines how far a functor $E$ is from being strongly $(p,q)$-polynomial. Lemma \ref{hofiblemma} is a $\C$-equivariant generalisation of \cite[Proposition 5.3]{Wei95}, and it describes how this fibre can be calculated. To prove this, we need the following $\C$-equivariant cofibre sequence, the proof of which follows that of the non-equivariant setting (see \cite[Proposition 5.4]{BO13}).

\begin{prop}\label{sphere cofibre seq}
For all $V,W\in \Jzero$ there is a homotopy cofibre sequence in $\C\TTop_*$
\begin{equation*}
    S\Gmor{p,q}{V}{W}_+\rightarrow\Jzero(V,W)\rightarrow\Jmor{p,q}{V}{W}.
\end{equation*}
\end{prop}

Using this proposition, we can prove that the functors $S \Gmor{p,q}{V}{-}_+$ and $\Jmor{p,q}{V}{-}$ are cofibrant in the projective model structure. 
\begin{lem}\label{sphereandmorpharecofibrant}
The functors $S \Gmor{p,q}{V}{-}_+$ and $\Jmor{p,q}{V}{-}$ are cofibrant objects in $\Ezero$. 
\end{lem}

\begin{proof}
The representable functor $\Jzero (V,-)$ is cofibrant by construction.

The homotopy limit used to construct $\taupq$ (see Lemma \ref{e.3}) preserves objectwise acyclic fibrations, since (indexed) products, totalization and the functor $\TTop_*(A,-)$, for a $\C$-CW complex $A$, all preserve acyclic fibrations in $\C\TTop_*$. It follows that $S\C\gamma_{p,q}(V,-)_+$ is cofibrant, by applying $\taupq$ to the diagram
% https://q.uiver.app/#q=WzAsNCxbMCwxLCJTXFxDXFxnYW1tYV97cCxxfShWLC0pXysiXSxbMSwxLCJGIl0sWzEsMCwiRSJdLFswLDAsIioiXSxbMywyXSxbMywwXSxbMiwxXSxbMCwxXV0=
\[\begin{tikzcd}
	{*} & E \\
	{S\C\gamma_{p,q}(V,-)_+} & F
	\arrow[from=1-1, to=1-2]
	\arrow[from=1-1, to=2-1]
	\arrow[from=1-2, to=2-2]
	\arrow[from=2-1, to=2-2]
\end{tikzcd}\]
where $E\rightarrow F$ is an objectwise acyclic fibration. Since, by Proposition \ref{sphere cofibre seq}, $\Jmor{p,q}{V}{-}$ is the cofibre of a map of cofibrant objects, it is also cofibrant. 
\end{proof}

The following result describes the relation between derivatives and polynomial functors. 

\begin{lem}\label{hofiblemma}
For all $E\in \Ezero$, and for all $V\in \Jzero$ there is a homotopy fibre sequence in $\C\TTop_*$
\begin{equation*}
    \ind_{0,0}^{p,q} E(V)\rightarrow E(V)\rightarrow\taupq E(V).
\end{equation*}
\end{lem}
\begin{proof}
By the previous proposition, there is a $\C$-equivariant homotopy cofibre sequence in $\Ezero$
\begin{equation*}
     S\Gmor{p,q}{V}{-}_+\rightarrow\Jzero(V,-)\rightarrow\Jmor{p,q}{V}{-}.
\end{equation*}

Pick $E\in\Ezero$ and apply the contravariant functor $\Nat_{p,q}(-,E)$ to the cofibre sequence above. This yields a homotopy fibre sequence (see Remark \ref{remNat})
\begin{equation*}
\Nat_{0,0}\left(S\Gmor{p,q}{V}{-}_+,E\right)\leftarrow \Nat_{0,0}\left(\Jzero(V,-),E\right)\leftarrow \Nat_{0,0}\left(\Jmor{p,q}{V}{-},E\right).
\end{equation*}
Application of the Yoneda Lemma and the definitions of  $\ind_{0,0}^{p,q}$ and $\taupq E$ gives the desired fibre sequence.
\end{proof}

From this fibration sequence, one can see that if a functor $E$ is strongly $(p,q)$-polynomial, then $\ind_{0,0}^{p,q}E(V)$ is contractible. This is analogous to an $n$-polynomial function having zero $(n+1)^{\text{st}}$ derivative. 

\begin{cor}\label{poly implies ind contractible}
If $E\in\Ezero$ is strongly $(p,q)$-polynomial, then $\ind_{0,0}^{p,q}E$ and $\ind_{0,0}^{p,q}\varepsilon^*E$ are objectwise contractible.  
\end{cor}

\subsection{Polynomial approximation}\label{sec: poly approx}

The partial sums of the Taylor series for a real function are known as the Taylor polynomials. These polynomial functions approximate the given function, and in general become better approximations as the degree of polynomial increases. In orthogonal calculus, Weiss defines a polynomial approximation functor $T_n$, see \cite[Theorem 6.3]{Wei95}. For an input functor $E$, each $T_nE$ is indeed an $n$-polynomial functor. In the $\C$-equivariant setting, we define an analogous functor $T_{p,q}$, and the $(p,q)$-polynomial approximation functor is given by the composition $T_{p+1,q}T_{p,q+1}$.

\begin{definition}\index{$\Tpq$}
Let $E\in \Ezero$, define the functor $\Tpq$ by
\begin{equation*}
    \Tpq E=\hocolim\left(E\overset{\rho_{p,q}}{\longrightarrow}\taupq E\overset{\rho_{p,q}}{\longrightarrow}\taupq^2 E\overset{\rho_{p,q}}{\longrightarrow}\dots\right),
\end{equation*}
\end{definition}
There is a natural transformation $\eta :E\rightarrow \Tpq E$.

\begin{ex}\label{exampleT10T01}
Let $E\in\Ezero$ be an input functor. Then
\begin{equation*}
    T_{1,0}E(V)=\hocolim_j \tau_{1,0}^j E(V)=\hocolim_j E(V\oplus \mathbb{R}^j),
\end{equation*}
\begin{equation*}
    T_{0,1}E(V)=\hocolim_j \tau_{0,1}^j E(V)=\hocolim_j E(V\oplus \mathbb{R}^{j\delta}),
\end{equation*}
and the strongly $(0,0)$-polynomial approximation is the constant functor $T_{0,0}E(V)=*$. 
\end{ex}

\begin{rem}
In orthogonal calculus, the $0$-polynomial approximation $T_0 F$ of an input functor $F\in\mathcal{E}_0$ is the constant functor taking value $F(\mathbb{R}^\infty)$. In particular, $T_0 F$ is a constant functor, but $T_{1,0} E$ and $T_{0,1}E$ are not in general. See Example \ref{ex: 00-poly approx} for the correct $\C$-analogue of $T_0 F$. 
\end{rem}

The functor $\Tpq E$ is strongly $(p,q)$-polynomial for all $E\in \Ezero$. To prove this, we need to generalise the erratum to orthogonal calculus \cite{Wei98} to the $\C$-equivariant setting. This is done over the following collection of lemmas. A formula similar to that used for the connectivity of $(\taupq s(W))^{\C}$ in Part 2 of the following lemma is used by Dotto in \cite[Corollary A.2]{Dot16b}. 
Recall that for a functor $G\in\Ezero$, the functor $\taupq G\in\Ezero$ is given by 
$$\taupq G(W)\cong \underset{0\neq U \subseteq \R^{p,q}}{\holim} \overline{G}(U\oplus W).$$ By abuse of notation, given $G\in \text{Fun}_{\TTop_*}(\Jzerobar,\C\TTop_*)$, we define the functor $\taupq G\in\Ezero$ by 
$$\taupq G(W):= \underset{0\neq U \subseteq \R^{p,q}}{\holim} G(U\oplus W).$$

Recall that a $\C$-space $Y$ is $v$-connected, where $v$ is a function from the set of conjugacy classes of closed subgroups of $\C$ to the integers, if each $X^H$ is a $v(H)$-connected space (see \cite[Definition 11.2.1]{May96}).

\begin{lem}\label{e.3}
Let $s:G\rightarrow F$ be a morphism in $\Fun_{\TTop_*}(\Jzerobar,\C\TTop_*)$ and $p,q\geq 0$. If there exists integers $b,c$ such that $s(W)$ is $v$-connected for all $W\in \Jzero$, where 
\begin{align*}
    v(e)&=(p+q)\Dim W -b\\
    v(\C)&=\Min\{(p+q)\Dim W -b, p\Dim W^{\C} + q\Dim (W^{\C})^\perp -c\},
\end{align*}
then $\taupq s(W)$ is $(v+1)$-connected.

\end{lem}
\begin{proof}
Let $\mathcal{D}$ denote the topological poset of non-zero linear subspaces of $\R^{p+q\delta}$. Similar to \cite[Lemma e.3]{Wei98}, the homotopy limit in $\taupq s(W)$ is the totalization of a cosimplicial object as follows. For $Z:\mathcal{D}\rightarrow \TTop_*$, the homotopy limit of $Z$ is the totalization of the cosimplicial space $$ [k]\mapsto \prod\limits_{L:[k]\rightarrow \mathcal{D}} Z(L(K))$$ taken over all monotone injections $[k]\rightarrow \mathcal{D}$. In particular, we are interested in the cases where $Z(U):= G(U\oplus W)$ and $Z(U):=F(U\oplus W)$.

Note that $\mathcal{D}$ is a category internal to $\C$-spaces. The space of objects is given by a disjoint union of $\C$-Grassmann manifolds 
\begin{equation*}
    \coprod\limits_{0\leq i\leq p+q} \mathcal{L}(\mathbb{R}^i,\mathbb{R}^{p+q\delta}) / O(i)
\end{equation*}
and the space of morphisms is the space of flags of subspaces of $\mathbb{R}^{p+q\delta}$ of length two
\begin{equation*}
    \coprod\limits_{0\leq i\leq j\leq p+q} \mathcal{L}(\mathbb{R}^j,\mathbb{R}^{p+q\delta}) / O(j-i)\times O(i),
\end{equation*}
where $\C$ acts by conjugation. 

To capture this structure, we can replace the cosimplicial space above by another. Let $\xi$ be the fibre bundle over $\mathcal{D}$, 
\begin{equation*}
    \coprod\limits_{0\leq i\leq p+q} Z(\R^i)\times_{O(i)}\LL(\R^i,\R^{p,q})\overset{\text{proj}}{\rightarrow} \coprod\limits_{0\leq i\leq p+q} \mathcal{L}(\mathbb{R}^i,\mathbb{R}^{p+q\delta}) / O(i)=\text{ob}(D),
\end{equation*}
such that the fibre of $U\in\mathcal{D}$ is $Z(U)$. Let $e_k: \mathcal{C}\rightarrow \mathcal{D}$ be defined by $L\mapsto L(k)$, where $\mathcal{C}$ is the $\C$-space of monotone injections $[k]\rightarrow \mathcal{D}$. We can replace the previous cosimplicial space by $$[k]\mapsto \Gamma (e_k^*\xi),$$ where $e_k^*\xi$ is the pullback bundle over $\mathcal{C}$ and $\Gamma$ denotes taking the section space. 

The space $\mathcal{C}$, of monotone injections $[k]\rightarrow \mathcal{D}$, is the disjoint union of $\C$-manifolds $C(\lambda)$, taken over monotone injections $\lambda:[k]\rightarrow [p+q]$ that avoid $0\in [p+q]$. These manifolds are defined by $$C(\lambda)=\{L:[k]\rightarrow\mathcal{D} : \Dim(L(i))=\lambda(i),\space \forall i\}.$$
That is, $C (\lambda)$ is the space of all flags of length $k$ and weight $\lambda$. Writing this as a quotient of orthogonal groups (where $\lambda_i=\lambda(i)$)
\begin{equation*}
    C(\lambda)= \coprod\limits_{0\leq \lambda_0\leq ...\leq \lambda_k\leq p+q} \mathcal{L}(\mathbb{R}^{\lambda_k},\mathbb{R}^{p+q\delta}) / O(\lambda_k-\lambda_{k-1})\times O(\lambda_{k-1}-\lambda_{k-2})\times ...\times O(\lambda_{0}),
\end{equation*}
one can calculate the dimension of $C(\lambda)$. 
\begin{align*}
    \Dim(C(\lambda))&=((p+q)-\lambda(k))\lambda(k) +\sum\limits_{i=0}^{k-1} (\lambda(i+1)-\lambda(i))\lambda(i)\\
    &= (p+q)\lambda(k) +\sum\limits_{i=0}^{k-1}\lambda(i)\lambda(i+1) -\sum\limits_{i=0}^{k}\lambda(i)^2\\
    &< (p+q)\lambda(k) -k
\end{align*}

Using the discussion of the homotopy limit as the totalization above, we see that the connectivity of $(\taupq s(W))^e$ is greater than or equal to the minimum of $$\conn(s(L(k)\oplus W))-\Dim((C(\lambda)) -k$$ taken over triples $(L,\lambda,k)$ with $L\in C(\lambda)$ and $\lambda:[k]\rightarrow [p+q]$. Substituting in the hypothesis on the connectivity of $s(L(k)\oplus W)$ and the bound on the dimension of $C(\lambda)$ yields that the connectivity of $(\taupq s(W))^e$ is at least $v(e)+1$. 

Note that for $G$-spaces $A,B$   
$$\conn (\TTop_*(A,B)^G)>\underset{H\leq G}{\text{min}}\{\conn B^H-\Dim A^H\}$$
taken over closed subgroups $H$ of $G$. Using this, along with the fact that fixed points commute with totalization, we see that the connectivity of $(\taupq s(W))^{\C}$ is greater than or equal to the minimum of 
\begin{equation*}
    \underset{H\leq \C}{\text{min}}\{\conn(s(L(k)\oplus W)^{H})-\Dim((C(\lambda)^{H}) -k\}
\end{equation*}
taken over triples $(L,\lambda,k)$ with $L\in C(\lambda)$ and $\lambda:[k]\rightarrow [p+q]$. A similar formula is used by Dotto in \cite[Corollary A.2]{Dot16b}.

One can determine $C(\lambda)^{\C}$ by applying the splitting theorem, Theorem \ref{splittingtheorems}, to the definition of $C(\lambda)$. Then a calculation similar to that above for $\Dim(C(\lambda))$ shows that $$\Dim(C(\lambda)^{\C}) < p\lambda(k)^{\C} +q(\lambda(k)^{\C})^\perp -k,$$ where $\lambda(k)^{\C}= \Dim(L(k)^{\C})$ and $(\lambda(k)^{\C})^{\perp}= \Dim((L(k)^{\C})^{\perp})$. The result then follows by substituting in the hypothesis on the connectivity of $s(L(k)\oplus W)^{\C}$ as we did for the first result above. 
\end{proof}

The following corollary can be proved using the same method as Lemma \ref{e.3}. 

\begin{cor}\label{erratum corollary}Let $s:G\rightarrow F$ be a morphism in $\Fun_{\TTop_*}(\Jzerobar,\C\TTop_*)$ and $p,q\geq 1$.
\begin{enumerate}
    \item If there exists integers $b,c$ such that $s(W)$ is $v$-connected for all $W\in \Jzero$, where 
        \begin{align*}
            v(e)&=2(p+q)\Dim W -b\\
            v(\C)&=\Min\{2(p+q)\Dim W -b, 2p\Dim W^{\C} + 2q\Dim (W^{\C})^\perp -c\},
        \end{align*}
        then $\tau_{p+1,q}\tau_{p,q+1} s(W)$ is at least $(v+1)$-connected.
    \item If there exists integers $b,c$ such that $s(W)$ is $v$-connected for all $W\in \Jzero$, where 
        \begin{align*}
            v(e)&=2(p+q)\Dim W -b\\
            v(\C)&=\Min\{2(p+q)\Dim W -b, (p+q)\Dim W -c\},
        \end{align*}
        then $\tau_{p+1,q}\tau_{p,q+1} s(W)$ is at least $(v+1)$-connected.
\end{enumerate}
\end{cor}

\begin{rem}
    Let $F,G\in\Ezero$. Recall that $\overline{G},\overline{F}$ are the right Kan extensions of $F,G$ respectively along the inclusion $\Jzero\rightarrow\Jzerobar$. If the map $s(W):F(W)\rightarrow G(W)$ is $v$-connected for all $W\in\Jzero$, then the map $\overline{s}(V):\overline{F}(V)\rightarrow \overline{G}(V)$ is $v$-connected for all $V\in\Jzerobar$. In particular, this means that given the connectivity of the map $s(W)$, Lemma \ref{e.3} can be applied to the map $\overline{s}(V)$ to make conclusions about the connectivity of the map $\taupq s(W)=\taupq \overline{s}(W)$.
\end{rem}

The following Lemma is a $\C$-version of \cite[Lemma e.7]{Wei98}. 

\begin{lem}\label{e.7}
Let $G:= S\Gmor{p,q}{V}{-}$, $F:=\Jzero(V,-)$, and let $s:G\rightarrow F$ be the projection sphere bundle map. $\Tpq s$ is an objectwise weak equivalence. 
\end{lem}

\begin{proof}
The map $s(W)$ satisfies the hypothesis of Lemma \ref{e.3} with 
\begin{align*}
    b&=(p+q)\Dim V+1\\
    c&=p\Dim V^{\C}+q\Dim (V^{\C})^{\perp}+1.
\end{align*}
Repeated application of Lemma $\ref{e.3}$ shows that the connectivity of both $(\taupq ^l s(W))^e$ and $(\taupq ^l s(W))^{\C}$ tend to infinity as $l$ tends to infinity. Thus, $(\Tpq s(W))^e$ and $(\Tpq s(W))^{\C}$ are weak homotopy equivalences, which is exactly that $\Tpq s$ is an objectwise weak equivalence. 
\end{proof}

The following Theorem is the statement that $\Tpq E$ is indeed strongly $(p,q)$-polynomial. The proof follows in the same way as the underlying statement \cite[Theorem 6.3.1]{Wei98}, using Lemma \ref{e.3} and Lemma \ref{e.7} in place of their non-equivariant counterparts. 

\begin{thm}\label{weiss6.3.1}
$\Tpq E$ is strongly $(p,q)$-polynomial for all $E\in\Ezero$ and all $p,q\geq 0$.
\end{thm}

The following is the $\C$-equivariant generalisation of \cite[Theorem 6.3.2]{Wei95}. The lemma demonstrates another property that one might expect strongly $(p,q)$-polynomial functors to satisfy based on the properties of polynomial functions. That is, the strongly $(p,q)$-polynomial approximation of a strongly $(p,q)$-polynomial functor is the functor itself. 

\begin{lem}\label{weiss6.3.2}
If $E$ is strongly $(p,q)$-polynomial, then $\eta: E\rightarrow\Tpq E$ is an objectwise weak equivalence. 
\end{lem}
\begin{proof}
If $E$ is strongly $(p,q)$-polynomial, then by definition $\rho^H:E(V)^H\rightarrow (\taupq E(V))^H$ is a weak homotopy equivalence, for all $V\in \Jzero$ and all $H$ closed subgroups of $\C$. Therefore, the map $E(V)^H\rightarrow (\hocolim_k \taupq^k E(V))^H$ is a weak homotopy equivalence, since fixed points commute with sequential homotopy colimits, which is exactly the map $\eta$. 
\end{proof}

Combining Lemma \ref{weiss6.3.2} with Theorem \ref{weiss6.3.1} gives the following Corollary. 

\begin{cor}\label{cor: TpqTpq we to Tpq}
Let $E\in\Ezero$, then $\Tpq E$ is objectwise weakly equivalent to $\Tpq \Tpq E$. 
\end{cor}

The following Lemma is a $\C$-equivariant version \cite[Lemma 5.11]{Wei95}. It says that $\tau_{l,m}$ preserves strongly $(p,q)$-polynomial functors. 
\begin{lem}\label{lem: 5.11}
If $E$ is strongly $(p,q)$-polynomial, then so is $\tau_{l,m}E$ for all $l,m\geq 0$. 
\end{lem}
\begin{proof}
Since homotopy limits commute and $\tau_{l,m}$ preserves objectwise weak equivalences we get the following. 
\begin{align*}
    \tau_{p,q}\tau_{l,m}E(V) &= \tau_{l,m}\taupq E(V)\\
    &\simeq \tau_{l,m}E(V)\qedhere
\end{align*}
\end{proof}

\begin{cor}\label{Tlm is pq poly if E is}
If $E$ is strongly $(p,q)$-polynomial, then so is $T_{l,m}E$ for all $l,m\geq 0$.    
\end{cor}
\begin{proof}
This is clear from Lemma \ref{lem: 5.11} using that $\taupq$ commutes with sequential homotopy colimits. 
\end{proof}

In particular, this allows us to define the $(p,q)$-polynomial approximation functor. 

\begin{definition}
Define the \emph{$(p,q)$-polynomial approximation} of an input functor $E\in \Ezero$ to be the functor $T_{p+1,q}T_{p,q+1}E$. By Theorem \ref{weiss6.3.1} and Corollary \ref{Tlm is pq poly if E is}, this is indeed a $(p,q)$-polynomial functor. 
\end{definition}

\begin{ex}\label{ex: 00-poly approx}
The $(0,0)$-polynomial approximation of a functor $E$ is the constant functor 
\begin{equation*}
    T_{1,0}T_{0,1}E(V)\cong \underset{k}{\hocolim}E(\R^{k,k})\cong \underset{a,b}{\hocolim}E(\R^{a,b})=: E(\R^{\infty,\infty}).
\end{equation*}
This is analogous to the $0$-polynomial approximation of an input functor $E$ being the constant functor $T_0E(V)=E(\R^\infty)$ in the underlying calculus. 
\end{ex}

Combining Lemma \ref{weiss6.3.2} and Corollary \ref{Tlm is pq poly if E is}, extends the result of Lemma \ref{weiss6.3.2} from strongly polynomial functors to polynomial functors. That is, the $(p,q)$-polynomial approximation of a $(p,q)$-polynomial functor is the functor itself. 
\begin{lem}
If $E$ is $(p,q)$-polynomial, then $E\simeq T_{p+1,q}T_{p,q+1}E$. \qed
\end{lem}

\subsection{The $(p,q)$-polynomial model structure}
Similar to Barnes and Oman \cite{BO13}, we would like to construct a model structure on the input category $\Ezero$ (see Definition \ref{jzero and ezero def}), that captures the homotopy theory of polynomial functors. We will construct the $(p,q)$-polynomial model structure on $\Ezero$ whose fibrant objects are functors that are $(p,q)$-polynomial. We construct this model structure, using the same method as Barnes and Oman in \cite[Section 6]{BO13}, by Bousfield-Friedlander localisation and left Bousfield localisation. To do this, we will also need the projective model structure on $\Ezero$ defined in Proposition \ref{proj model structure}.

To begin, we will construct a model structure on $\Ezero$ whose fibrant objects are the strongly $(p,q)$-polynomial functors. This model structure will allow us to easily deduce results about strongly $(p,q)$-polynomial functors, without having to keep track of the more complex indexing of the $(p,q)$-polynomial model structure.  

\begin{definition}
A morphism $f\in \Ezero$ is a \emph{$\Tpq$-equivalence} if $\Tpq f$ is an objectwise weak equivalence (see Definition \ref{def: obj WE}). 
\end{definition}

\begin{prop}\label{Tpq model structure}\index{$(p,q)\poly\Ezero^S$}
There is a proper model structure on $\Ezero$ such that a morphism $f$ is a weak equivalence if and only if it is a $\Tpq$-equivalence. The cofibrations are the same as for the projective model structure. The fibrant objects are the strongly $(p,q)$-polynomial functors. A morphism $f$ is a fibration if and only if it is an objectwise fibration and the diagram 
% https://q.uiver.app/#q=WzAsNCxbMCwwLCJYIl0sWzIsMCwiWSJdLFsyLDIsIlRfe3AscX1ZIl0sWzAsMiwiVF97cCxxfVgiXSxbMCwxLCJmIl0sWzAsMywiXFxldGEiLDJdLFsxLDIsIlxcZXRhIl0sWzMsMiwiVF97cCxxfWYiLDJdXQ==
\[\begin{tikzcd}
	X && Y \\
	\\
	{T_{p,q}X} && {T_{p,q}Y}
	\arrow["f", from=1-1, to=1-3]
	\arrow["\eta"', from=1-1, to=3-1]
	\arrow["\eta", from=1-3, to=3-3]
	\arrow["{T_{p,q}f}"', from=3-1, to=3-3]
\end{tikzcd}\]
\noindent is a homotopy pullback square in $\Ezero$. Denote this model structure by $(p,q)\poly\Ezero^S$. 
\end{prop}

\begin{proof}
The model category $(p,q)\poly\Ezero^S$ is the Bousfield-Friedlander localisation of $\Ezero$ with respect to the functor $T_{p,q}:\Ezero\rightarrow \Ezero$, since $T_{p,q}$ satisfies the conditions of \cite[Theorem 9.3]{Bou01}. Moreover, $(p,q)\poly\Ezero^S$ is the left Bousfield localisation of $\Ezero$ with respect to the class of maps 
\begin{equation*}
    S_{p,q}=\{S\Gmor{p,q}{V}{-}_+\rightarrow\Jzero(V,-):V\in\Jzero\}.
\end{equation*}
The proof is analogous to that of Barnes and Oman \cite[Proposition 6.6]{BO13}, where is sufficient to show that the fibrant objects of the left Bousfield localisation are the strongly $(p,q)$-polynomial functors, since both classes of cofibrations are the same.
\end{proof}

\begin{cor}\label{tpqisspq}
The class of $\Tpq$-equivalences is the collection of $S_{p,q}$-local equivalences. 
\end{cor}

Using this model structure, one can prove that a strongly $(p,q)$-polynomial functor is indeed $(p,q)$-polynomial. The underlying version is proven in by Weiss in \cite[Proposition 5.4]{Wei95} and by Barnes and Oman in \cite[Proposition 6.7]{BO13}. 

\begin{prop}\label{pq poly implies more poly}
If $X\in \Ezero$ is strongly $(p,q)$-polynomial, then it is $(p,q)$-polynomial.
\end{prop}
\begin{proof}
Since the weak equivalences in $(p,q)$-poly-$\Ezero^S$ are the $S_{p,q}$-local equivalences (see Corollary \ref{tpqisspq}), it suffices to show that $S_{p+1,q}$-equivalences and $S_{p,q+1}$-equivalences are $S_{p,q}$-equivalences. This is sufficient, since it proves the existence of Quillen adjunctions
\begin{align*}
    &\Id:(p+1,q)\poly\Ezero^S \rightleftarrows (p,q)\poly\Ezero^S:\Id\\
      &\Id:(p,q+1)\poly\Ezero^S \rightleftarrows (p,q)\poly\Ezero^S:\Id,
\end{align*}
and $\Id$ being right Quillen yields the desired result. The proof is analogous to \cite[Proposition 6.7]{BO13}, using that the unit sphere of a Whitney sum of $\C$-vector bundles is $\C$-equivariantly equivalent to the fibrewise join of the unit sphere bundles. 
\end{proof}

Combining Proposition \ref{pq poly implies more poly} with Theorem \ref{weiss6.3.1} and Lemma \ref{weiss6.3.2}, gives an important result about how the functors $\Tpq$ interact. 

\begin{cor}\label{tlmtpq=tpq}
Let $E\in\Ezero$. If $l\geq p$ and $m\geq q$, then $$T_{p,q}T_{l,m}E\simeq T_{l,m}T_{p,q}E\simeq T_{p,q}E.$$
\end{cor}

\begin{prop}\label{polynomial model structure}\index{$(p,q)\poly\Ezero$}
There is a proper model structure on $\Ezero$ such that a morphism $f$ is a weak equivalence if and only if it is a $T_{p+1,q}T_{p,q+1}$-equivalence. The cofibrations are the same as for the projective model structure. The fibrant objects are the functors that are $(p,q)$-polynomial. A morphism $f$ is a fibration if and only if it is an objectwise fibration and the diagram 
% https://q.uiver.app/#q=WzAsNCxbMCwwLCJYIl0sWzIsMCwiWSJdLFswLDIsIlRfe3ArMSxxfVRfe3AscSsxfVgiXSxbMiwyLCJUX3twKzEscX1UX3twLHErMX1ZIl0sWzIsMywiVF97cCsxLHF9VF97cCxxKzF9ZiIsMl0sWzAsMSwiZiJdLFsxLDMsIlxcZXRhIl0sWzAsMiwiXFxldGEiLDJdXQ==
\[\begin{tikzcd}
	X && Y \\
	\\
	{T_{p+1,q}T_{p,q+1}X} && {T_{p+1,q}T_{p,q+1}Y}
	\arrow["{T_{p+1,q}T_{p,q+1}f}"', from=3-1, to=3-3]
	\arrow["f", from=1-1, to=1-3]
	\arrow["\eta", from=1-3, to=3-3]
	\arrow["\eta"', from=1-1, to=3-1]
\end{tikzcd}\] is a homotopy pullback square in $\Ezero$. Denote this model structure by $(p,q)\poly\Ezero$.
\end{prop}

\begin{proof}The model category $(p,q)\poly\Ezero$ is the Bousfield-Friedlander localisation of $\Ezero$ with respect to the functor $T_{p+1,q}T_{p,q+1}:\Ezero\rightarrow \Ezero$, since $T_{p+1,q}T_{p,q+1}$ satisfies the conditions of \cite[Theorem 9.3]{Bou01}. Moreover, $(p,q)\poly\Ezero$ is the left Bousfield localisation of $\Ezero$ with respect to the class of maps $S_{p+1,q}\coprod S_{p,q+1}$ where 
\begin{align*}
    S_{p+1,q}&=\{S\Gmor{p+1,q}{V}{-}_+\rightarrow\Jzero(V,-):V\in\Jzero\}\\
    S_{p,q+1}&=\{S\Gmor{p,q+1}{V}{-}_+\rightarrow\Jzero(V,-):V\in\Jzero\}.
\end{align*}
\end{proof}

\section{Equivariant Homogeneous functors}

In orthogonal calculus, the homotopy fibre of the map $T_n F\rightarrow T_{n-1} F$ is $n$-polynomial and has trivial $(n-1)$-polynomial approximation. Functors of this type are called $n$-homogeneous, and are completely determined by orthogonal spectra with $O(n)$-action, see \cite[Theorem 7.3]{Wei95}. In this section we define a new class of $(p,q)$-homogeneous functors in the $\C$-equivariant input category. The main goal will be to classify such functors by a category of spectra, as is done in the underlying calculus. The relation between the $(p,q)$-homogeneous model structure and the stable model structure in Proposition \ref{prop: stable ms}, forms one half of the zig-zag of equivalences that gives this classification, see Theorem \ref{zigzagclassification}.

\subsection{Homogeneous functors}\label{sec: homog functors}

We want to use the properties of the homotopy fibre $$D_{p,q} F\rightarrow T_{p+1,q}T_{p,q+1}F\rightarrow \Tpq F$$to define a class of $(p,q)$-homogeneous functors in the input category $\Ezero$ (see Definition \ref{jzero and ezero def}). First, we will need to determine what these properties are, and to do so we will need to give some important and useful properties of polynomial functors. Many of these properties will also be needed when constructing the homogeneous model structure in Section \ref{sec:homog ms}. 

The following is the $\C$-version of \cite[Lemma 5.5]{Wei95}. The proof follows, as in \cite{Wei95}, from the homotopy fibre sequence in Lemma \ref{hofiblemma}. 

\begin{lem}\label{lem: c2 wes 5.5}
Let $g:E\rightarrow G$ in $\Ezero$ be such that $\ind_{0,0}^{p,q}G$ is objectwise contractible and $E$ is strongly $(p,q)$-polynomial, then the homotopy fibre of $g$ is strongly $(p,q)$-polynomial. 
\end{lem}

\begin{cor}\label{cor: c2 weiss 5.5}
Let $g:E\rightarrow F$ in $\Ezero$ be such that $\ind_{0,0}^{p+1,q}F$ and $\ind_{0,0}^{p,q+1}F$ are both objectwise contractible and $E$ is $(p,q)$-polynomial, then the homotopy fibre of $g$ is $(p,q)$-polynomial. \qed
\end{cor}

The following is a $\C$-generalisation of \cite[Corollary 5.6]{Wei95}, which is an instant consequence Corollary \ref{cor: c2 weiss 5.5} by setting $E=*$. 
\begin{cor}\label{cor: F delooping}
Let $F\in\Ezero$ be such that $\ind_{0,0}^{p+1,q}F$ and $\ind_{0,0}^{p,q+1}F$ are both objectwise contractible, then the functor 
\begin{equation*}
    V\mapsto \Omega F(V)
\end{equation*}
is $(p,q)$-polynomial. \qed
\end{cor}

Now we can determine the properties of the homotopy fibre $D_{p,q}F$. 

\begin{definition}\label{def: pq reduced}
Let $E\in \Ezero$. Define $E$ to be \emph{$(p,q)$-reduced} if $\Tpq E$ is objectwise contractible. 
\end{definition}

\begin{rem}
    Note that if $E\in\Ezero$ is $(p,q)$-reduced, then $E$ is also $(a,b)$-reduced for all pairs $(a,b)$ with $0\leq a\leq p$ and $0\leq b\leq q$. This follows from Corollary \ref{tlmtpq=tpq}.
\end{rem}

\begin{thm}\label{thm: DYpq is homog}\index{$D_{p,q}$}
The homotopy fibre $D_{p,q} F$ of the map $r_{p,q}: T_{p+1,q}T_{p,q+1} F\rightarrow T_{p,q} F$ is $(p,q)$-polynomial and $(p,q)$-reduced.
\end{thm}
\begin{proof}
Indeed, applying $\Tpq$ to the homotopy fibre sequence and applying Corollary \ref{tlmtpq=tpq} gives
% https://q.uiver.app/#q=WzAsNixbMCwwLCJUX3twLHF9RF97cCxxfV5ZIEYiXSxbMSwwLCJUX3twLHF9VF97cCsxLHF9VF97cCxxKzF9RiJdLFsyLDAsIlRfe3AscX1UX3twLHF9RiJdLFsxLDEsIlRfe3AscX1GIl0sWzIsMSwiVF97cCxxfUYiXSxbMCwxLCIqIl0sWzAsMV0sWzEsMl0sWzEsMywiXFxzaW1lcSIsMyx7InN0eWxlIjp7ImJvZHkiOnsibmFtZSI6Im5vbmUifSwiaGVhZCI6eyJuYW1lIjoibm9uZSJ9fX1dLFsyLDQsIlxcc2ltZXEiLDMseyJzdHlsZSI6eyJib2R5Ijp7Im5hbWUiOiJub25lIn0sImhlYWQiOnsibmFtZSI6Im5vbmUifX19XSxbMyw0XSxbNSwzXV0=
\[\begin{tikzcd}
	{T_{p,q}D_{p,q} F} & {T_{p,q}T_{p+1,q}T_{p,q+1}F} & {T_{p,q}T_{p,q}F} \\
	& {T_{p,q}F} & {T_{p,q}F}
	\arrow[from=1-1, to=1-2]
	\arrow["r_{p,q}", from=1-2, to=1-3]
	\arrow["\simeq"{marking}, draw=none, from=1-2, to=2-2]
	\arrow["\simeq"{marking}, draw=none, from=1-3, to=2-3]
	\arrow["\id"',from=2-2, to=2-3]
\end{tikzcd}\]
Therefore, $\Tpq D_{p,q}F\simeq \hofibre[\Tpq F\overset{\id}{\rightarrow} \Tpq F]\simeq *$. That is, $ D_{p,q}F$ is $(p,q)$-reduced. 

The functors $T_{p+1,q}T_{p,q+1}F$ and $\Tpq F$ are both $(p,q)$-polynomial. This follows from Theorem \ref{weiss6.3.1} and Proposition \ref{pq poly implies more poly}. Therefore, $D_{p,q} F$ is $(p,q)$-polynomial, by Corollary \ref{cor: c2 weiss 5.5}.
\end{proof}

\begin{definition}
Let $E\in \Ezero$. $E$ is defined to be \emph{$(p,q)$-homogeneous} if it is $(p,q)$-polynomial and $(p,q)$-reduced.
\end{definition}

We now define the equivariant generalisation of the term `connected at infinity', this is analogous to \cite[Definition 5.9]{Wei95}. 

\begin{definition}
$E\in\Ezero$ is defined to be \emph{connected at infinity} if the $\C$-space 
\begin{equation*}
    \underset{a,b}{\hocolim} E(\R^{a,b})=:E(\R^{\infty,\infty})
\end{equation*}
is connected (meaning that the equivariant homotopy groups $\pi_0^H E(\R^{\infty,\infty})$ are trivial for all closed subgroups $H$ of $\C$). 
\end{definition}

\begin{lem}\label{lem: homog implies connected at infinity}
If $E$ is $(p,q)$-homogeneous for pairs $(p,q)$ where at least one of $p,q$ is greater than zero, then $E$ is connected at infinity. 
\end{lem}
\begin{proof}
This is a straightforward calculation that follows from Example \ref{ex: 00-poly approx}. 
\end{proof}
 
The next result is \cite[Proposition 5.10]{Wei95}. It is used in conjunction with Lemma \ref{lem: homog implies connected at infinity} to construct the $(p,q)$-homogeneous model structure. The proof follows as in \cite[Lemma 5.10]{BO13}, replacing $T_n$ by $T_{p,q}$, so we omit it here.

\begin{lem}\label{conn at infinity lemma}
Let $g:E\rightarrow F$ be a map between strongly $(p,q)$-polynomial objects such that the homotopy fibre of $g$ is objectwise contractible and $F$ is connected at infinity. Then $g$ is an objectwise weak equivalence. 
\end{lem}

Finally we will prove a version of \cite[Corollary 5.12]{Wei95}. We will use this result later, along with Lemma \ref{omega spectra lemma}, to show that the induction functor takes objects that are $(p,q)$-polynomial to $(p,q)\Omega$-spectra. This guarantees that induction is a right Quillen functor from the $(p,q)$-polynomial model structure on $\Ezero$ to the $(p,q)$-stable model structure on $\Epq$.

\begin{prop}\label{5.12}
Let $E\in \Ezero$ be $(p,q)$-polynomial. Then for all $V\in\Jzero$ there exist weak equivalences of $\C$-spaces
\begin{align*}
    \ind_{0,0}^{p,q}E(V)&\rightarrow\Omega^{p,q\R}\ind_{0,0}^{p,q}E(V\oplus\R)\\
    \ind_{0,0}^{p,q}E(V)&\rightarrow\Omega^{p,q\Rdelta}\ind_{0,0}^{p,q}E(V\oplus\Rdelta).
\end{align*}
\end{prop}

\begin{proof}
We will prove the first weak equivalence, since the second follows by a similar argument.

If $p=q=0$, then there is nothing to prove, since $E$ is constant and $\ind_{0,0}^{0,0}E\simeq E$ by the enriched Yoneda lemma. 

Let $p,q$ be such that at least one of $p,q$ is non-zero. By Proposition \ref{loops fibre sequence}, there exists a $\C$-homotopy fibre sequence
\begin{equation*}
    \res_{p,q}^{p+1,q}\ind_{0,0}^{p+1,q}E(V)\rightarrow \ind_{0,0}^{p,q}E(V)\rightarrow\Omega^{p,q\R}\ind_{0,0}^{p,q}E(V\oplus\R).
\end{equation*}
We know that $\ind_{0,0}^{p+1,q}E(V)$ is a contractible $\C$-space, since $E$ is strongly $(p+1,q)$-polynomial and by Corollary \ref{poly implies ind contractible}. Thus, if we can show that both $\ind_{0,0}^{p,q}E$ and $$F:V\mapsto \Omega^{p,q\R}\ind_{0,0}^{p,q}E(V\oplus\R)$$ are strongly $(p+1,q)$-polynomial and $F$ is connected at infinity, then Lemma \ref{conn at infinity lemma} gives the weak equivalence. Showing this is the same as for \cite[Proposition 5.12]{BO13}.
\end{proof}

\subsection{The $(p,q)$-homogeneous model structure}\label{sec:homog ms}
We now construct a right Bousfield localisation of the $(p,q)$-polynomial model structure in order to build a model structure on $\Ezero$ analogous to the $n$-homogeneous structure in \cite[Proposition 6.9]{BO13}. The cofibrant-fibrant objects of this model structure are the projectively cofibrant $(p,q)$-homogeneous functors and the weak equivalences are detected by derivatives. This model structure allows for the classification of $(p,q)$-homogeneous functors in terms of a Quillen equivalence, see Theorem \ref{boclassification}. 

\begin{prop}\index{$(p,q)\homog\Ezero$}
There exists a model structure on $\Ezero$ whose cofibrant-fibrant objects are the $(p,q)$-homogeneous functors that are cofibrant in the projective model structure on $\Ezero$. Fibrations are the same as $(p,q)\poly\Ezero$ and weak equivalences are morphisms $f$ such that $\res_{0,0}^{p,q}\ind_{0,0}^{p,q}T_{p+1,q}T_{p,q+1}f$ is an objectwise weak equivalence. 
We call this the $(p,q)$-homogeneous model structure on $\Ezero$ and denote it by $(p,q)\homog\Ezero$. 

\noindent There is a Quillen adjunction 
\begin{equation*}
    \Id: (p,q)\homog\Ezero\rightleftarrows (p,q)\poly\Ezero:\Id
\end{equation*}
\end{prop}

\begin{proof}Right Bousfield localisation of $(p,q)$-poly-$\Ezero$ with respect to the set of objects 
\begin{equation*}
    K_{p,q}=\{\Jpq(V,-):V\in \Jzero\}
\end{equation*}
yields the desired model structure. Since $(p,q)$-poly-$\Ezero$ is proper and cellular, we know that the right Bousfield localisation $R_{K_{p,q}}((p,q)\poly\Ezero)$ exists by \cite[Theorem 5.1.1]{Hir03}. 

The weak equivalences of $R_{K_{p,q}}((p,q)\poly\Ezero)$ are the $K_{p,q}$-colocal equivalences. That is, a morphism $f:X\rightarrow Y$ is a weak equivalence in $R_{K_{p,q}}((p,q)\poly\Ezero)$ if and only if 
\begin{equation*}
    \Nat_{0,0}(\Jpq(V,-),T_{p+1,q}T_{p,q+1}X)\rightarrow \Nat_{0,0}(\Jpq(V,-),T_{p+1,q}T_{p,q+1}Y)
\end{equation*}
is a weak equivalence in $\CTop_*$, for all $V\in\Jzero$. By Definition \ref{induction definition}, we see that this map is exactly $\res_{0,0}^{p,q}\ind_{0,0}^{p,q}T_{p+1,q}T_{p,q+1}f(V)$. Therefore, $f$ is a weak equivalence if and only if $\res_{0,0}^{p,q}\ind_{0,0}^{p,q}T_{p+1,q}T_{p,q+1}f$ is an objectwise weak equivalence as desired. 
\end{proof}

\begin{rem}
Detecting weak equivalences via $\ind_{0,0}^{p,q}T_{p+1,q}T_{p,q+1}$ can be difficult, since the induction functor $\ind_{0,0}^{p,q}$ is complex. In unitary calculus, Taggart shows that a map is an $\ind_{0}^n T_n$-equivalence if and only if it is a $D_n$-equivalence, where $D_nF$ is the homotopy fibre of $T_nF\rightarrow T_{n-1}F$ (see \cite[Proposition 8.2]{Tag22unit}). Via a similar proof, one can show that a map is an $\ind_{0,0}^{p,q}T_{p+1,q}T_{p,q+1}$-equivalence if and only if it is a $D_{p,q}$-equivalence.  
\end{rem}

\section{The equivariant classification theorem}
The main result of orthogonal calculus is the classification of $n$-homogeneous functors, as functors fully determined by a category of spectra. This is given by the classification theorem of Weiss \cite[Theorem 7.3]{Wei95}, which states that an $n$-homogeneous functor is weakly equivalent to a functor of the form $V\mapsto \Omega^\infty [(S^{nV}\wedge \Psi)_{hO(n)}]$, where $\Psi$ is an orthogonal spectrum with an action of $O(n)$. In \cite{BO13}, this classification is derived as a Quillen equivalence on the model categories constructed. As a result, the homotopy fibres of the maps $T_{n+1}X\rightarrow T_nX$ for an input functor $X$, which are $n$-homogeneous, stand a chance of begin computed. 

In this section, we construct two Quillen equivalences. These Quillen equivalences form a zig-zag of equivalences between the $(p,q)$-homogeneous model structure on $\Ezero$ and the stable model structure on $\C Sp^O[O(p,q)]$. 
% https://q.uiver.app/?q=WzAsMyxbMCwwLCIocCxxKVxcdGV4dHstaG9tb2ctfUNfMlxcbWF0aGNhbHtFfV97MCwwfSJdLFs0LDAsIkNfMlNwXlxcbWF0aGNhbHtPfVtPKHArcSldIl0sWzIsMCwiTyhwK3EpQ18yXFxtYXRoY2Fse0V9X3twLHF9XnMiXSxbMCwyLCJcXHRleHR7aW5kfV97MCwwfV57cCxxfVxcdmFyZXBzaWxvbl4qIiwyLHsib2Zmc2V0IjoyfV0sWzIsMCwiXFx0ZXh0e3Jlc31fezAsMH1ee3AscX0vTyhwK3EpIiwyLHsib2Zmc2V0IjoyfV0sWzIsMSwiKFxcYWxwaGFfe3AscX0pXyEiLDAseyJvZmZzZXQiOi0yfV0sWzEsMiwiXFxhbHBoYV97cCxxfV4qIiwwLHsib2Zmc2V0IjotMn1dXQ==
\[\begin{tikzcd}
	{(p,q)\homog C_2\mathcal{E}_{0,0}} && {O(p,q)C_2\mathcal{E}_{p,q}^s} && {C_2Sp^O[O(p,q)]}
	\arrow["{\ind_{0,0}^{p,q}\varepsilon^*}"', shift right=2, from=1-1, to=1-3]
	\arrow["{\res_{0,0}^{p,q}/O(p,q)}"', shift right=2, from=1-3, to=1-1]
	\arrow["{(\alpha_{p,q})_!}", shift left=2, from=1-3, to=1-5]
	\arrow["{\alpha_{p,q}^*}", shift left=2, from=1-5, to=1-3]
\end{tikzcd}\]

In the same way as for Barnes and Oman in \cite[Section 10]{BO13}, this leads to a classification theorem for $(p,q)$-homogeneous functors (Theorem \ref{weissclassification}), as functors fully determined by genuine orthogonal $\C$-spectra with an action of $O(p,q)$.

\subsection{The intermediate category as a category of spectra}

In this section we show that the $(p,q)$-th intermediate category $\OEpq$ is Quillen equivalent to the category of $O(p,q)$-equivariant objects in orthogonal $\C$-spectra. Thus, the $(p,q)$-derivative of an input functor (see Definition \ref{jzero and ezero def}) can be described in terms of these spectra. This section is analogous to \cite[Section 8]{BO13}, where the only differences are due to the additional $\C$-action, which does not effect the $O(p,q)$-equivariance of maps considered. The resulting Quillen equivalence 
\begin{equation*}
    (\alpha_{p,q})_{!}: \OEpq \rightleftarrows \C Sp^O [O(p,q)] : \alpha_{p,q}^*
\end{equation*}   
forms one half of the zig-zag of equivalences which gives the classification of $(p,q)$-homogeneous functors, see Theorem \ref{zigzagclassification}. The spectrum $(\alpha_{p,q})_{!}F$ for a functor $F\in \OEpq$ is the categorification of the spectrum $\Theta F$ constructed in \cite[Section 2]{Wei95}.

We begin by describing the category of orthogonal $\C$-spectra. Details of these constructions have been discussed by Mandell and May in \cite[Section II.4]{MM02}. 

\begin{definition}\label{def: orth c2 spectra}\index{$\C Sp^O$}
The \emph{category of orthogonal $\C$-spectra} $\C Sp^O$ is the category $\C\mathcal{E}_{1,0}$.

\noindent This category has a cofibrantly generated proper stable model structure where the cofibrations are $q$-cofibrations and the weak equivalences are the $(1,0)\pi_*$-equivalences (see Definition \ref{def: pq pi equiv}). 
\end{definition}

\begin{rem}
Sometimes in the literature, an orthogonal $G$-spectrum is defined as a collection of based spaces $\{X_n\}_{n\in \mathbb{N}}$ with an action of $G\times O(n)$ on each $X_n$. The structure maps are $G$-equivariant maps 
\begin{equation*}
    X_n\wedge S^1\rightarrow X_{n+1},
\end{equation*}
such that the iterated structure maps 
\begin{equation*}
    X_n\wedge S^m \rightarrow X_{n+m}
\end{equation*}
are $O(m)\times O(n)$-equivariant, where $G$ acts trivially on $S^m$. 

This is what we refer to as a naive orthogonal $G$-spectrum. Genuine $G$-spectra are indexed on a complete $G$-universe of all $G$-representations, whereas naive $G$-spectra are indexed on the trivial $G$-universe. That is, naive $G$-spectra are just spectra with an action of $G$ on each level, and $G$-equivariant structure maps. These two descriptions of $G$-spectra are categorically equivalent, however they are not homotopically equivalent (see \cite[Section 9.3]{HHR21}), in that the most natural model structures associated to these categories are not Quillen equivalent. In what follows, we consider orthogonal spectra that are genuine with respect to $\C$ and naive with respect to $O(p,q)$. 
\end{rem}

\begin{definition}\index{$\C Sp^O [O(p,q)]$}
The category of \emph{$O(p,q)$-objects in orthogonal $\C$-spectra}, $\C Sp^O [O(p,q)]$, is the category of $O(p,q)$-objects in $\C \mathcal{E}_{1,0}$ and $O(p,q)$-equivariant maps. An $O(p,q)$-object in $\C \mathcal{E}_{1,0}$ is a $\C\TTop_*$-enriched functor from $\C\mathcal{J}_{1,0}$ to $\COTop_*$. 
\end{definition}

\begin{thm}\label{thm: stable ms on C2O-spectra}
The category of genuine orthogonal $\C$-spectra with an action of $O(p,q)$, $\C Sp^O [O(p,q)]$, has a cofibrantly generated proper stable model structure where fibrations and weak equivalences are defined by the underlying model structure on $\C Sp^O$ above.
\end{thm}

We want to define a functor $\C Sp^O [O(p,q)]\rightarrow \OEpq$. This functor will be called $\alpha_{p,q}^*$, and it is analogous to the $\alpha_n^*$ functor built in \cite[Section 8]{BO13}.

\begin{definition}\index{$\alpha_{p,q}^*$}
Define the functor $\alpha_{p,q}:\Jpq\rightarrow\C\mathcal{J}_{1,0}$ by 
\begin{equation*}
    U\mapsto \mathbb{R}^{p,q}\otimes U = (p,q)U
\end{equation*}
on objects, and 
\begin{equation*}
    (f,x)\mapsto(\mathbb{R}^{p,q}\otimes f, x)
\end{equation*}
on morphisms. 
\end{definition}

Since the map on morphisms is $\left(O(p,q)\rtimes \C\right)$-equivariant, $\alpha_{p,q}$ is enriched over $\left(O(p,q)\rtimes \C\right)$-spaces. 

This induces a functor 
\begin{equation*}
    \alpha_{p,q}^*:\C Sp^O[O(p,q)]\rightarrow \OEpq
\end{equation*}
defined by precomposition with $\alpha_{p,q}$. For $X\in\C Sp^O[O(p,q)]$ we define the $ (O(p,q)\rtimes \C)$- action on $(\alpha_{p,q}X)(V):=X((p,q)V)$ by 
\begin{equation*}
 X(g\sigma \otimes \sigma)\circ (g\sigma )_{X((p,q)V)}
\end{equation*}Here $ X(g\sigma \otimes \sigma)$ is the internal action on $X((p,q)V)$ induced by the action on $(p,q)V$, and $(g\sigma )_{X((p,q)V)}$ is the external action from $X((p,q)V)$ being an $\left( O(p,q)\rtimes \C\right)$-space. These two actions commute by construction. 

Checking that $\alpha_{p,q}^*X$ is well defined (i.e. that $\alpha_{p,q}^*X$ is $(O(p,q)\rtimes \C)\TTop_*$-enriched) is the same as checking that the map 
\begin{equation*}
    \Jpq(U,V)\rightarrow \COTop_*(X((p,q)V),X((p,q)V))
\end{equation*}
is $\left(O(p,q)\rtimes \C\right)$-equivariant. 

To do this we consider the following commutative diagram. We use the notation $(-)^*$ to mean pre-composition and $(-)_*$ to mean post-composition. Let $s$ denote the map $((g\sigma )^{-1}\otimes \sigma)^*\circ (g\sigma  \otimes \sigma)_*$, and $t$ be the map $(X((g\sigma )^{-1}\otimes \sigma))^*\circ (X(g\sigma  \otimes \sigma))_*$. We have abbreviated $\COTop_*$ to $S\TTop_*$ to save space ($S$ for semi-direct product). 
% https://q.uiver.app/#q=WzAsNixbMCwwLCJcXENcXG1hdGhjYWx7Sn1fe3AscX0oVSxWKSJdLFsxLDAsIlxcQ1xcbWF0aGNhbHtKfV97MSwwfSgocCxxKVUsKHAscSlWKSJdLFsyLDAsIihPKHArcSlcXHJ0aW1lc1xcQylcXFRvcF8qKFgoKHAscSlVKSxYKChwLHEpVikpIl0sWzAsMiwiXFxDXFxtYXRoY2Fse0p9X3twLHF9KFUsVikiXSxbMSwyLCJcXENcXG1hdGhjYWx7Sn1fezEsMH0oKHAscSlVLChwLHEpVikiXSxbMiwyLCIoTyhwK3EpXFxydGltZXNcXEMpXFxUb3BfKihYKChwLHEpVSksWCgocCxxKVYpKSJdLFsyLDUsInQiXSxbMSw0LCJzIl0sWzAsMywiZ1xcc2lnbWEiXSxbMCwxLCJcXGFscGhhX3twLHF9Il0sWzMsNCwiXFxhbHBoYV97cCxxfSIsMl0sWzEsMiwiWCJdLFs0LDUsIlgiLDJdXQ==
\[\begin{tikzcd}
	{\C\mathcal{J}_{p,q}(U,V)} & {\C\mathcal{J}_{1,0}((p,q)U,(p,q)V)} & {S\TTop_*(X((p,q)U),X((p,q)V))} \\
	\\
	{\C\mathcal{J}_{p,q}(U,V)} & {\C\mathcal{J}_{1,0}((p,q)U,(p,q)V)} & {S\TTop_*(X((p,q)U),X((p,q)V))}
	\arrow["t", from=1-3, to=3-3]
	\arrow["s", from=1-2, to=3-2]
	\arrow["g\sigma", from=1-1, to=3-1]
	\arrow["{\alpha_{p,q}}", from=1-1, to=1-2]
	\arrow["{\alpha_{p,q}}"', from=3-1, to=3-2]
	\arrow["X", from=1-2, to=1-3]
	\arrow["X"', from=3-2, to=3-3]
\end{tikzcd}\]
Given a pair $(f,x)\in \Jmor{p,q}{U}{V}$, by applying $\alpha_{p,q}^* X$ we get a a $\left( O(p,q)\rtimes \C\right)$-equivariant map
\begin{equation*}
    X(\mathbb{R}^{p,q}\otimes f): X((p,q)V)\rightarrow X((p,q)V).
\end{equation*}Therefore, the following two expressions are equal 
\begin{equation*}
    X(g\sigma \otimes V)\circ X(\mathbb{R}^{p,q}\otimes f,x)\circ X((g\sigma )^{-1}\otimes U)
\end{equation*}
\begin{equation*}
   (g\sigma )_{X(U)}\circ X(g\sigma \otimes V)\circ X(\mathbb{R}^{p,q}\otimes f,x)\circ X((g\sigma )^{-1}\otimes U)\circ (g\sigma )^{-1}_{X(V)}
\end{equation*}which by comparison to the commutative diagram tells us exactly that the map 
\begin{equation*}
    \Jpq(U,V)\rightarrow \COTop_*(X((p,q)V),X((p,q)V))
\end{equation*}
is $\left(O(p,q)\rtimes \C\right)$-equivariant, and hence that $\alpha_{p,q}^*X$ is well defined. 

\begin{rem}
Note that any other choice of internal action on $X((p,q)V)$ results in the failure of the diagram being commutative. For example, taking $X(g\sigma \otimes \id)$ as the internal action means that the left square in the diagram commutes only if $f$, from the pair $(f,x)$, is $\C$-equivariant, which is not necessarily the case.
\end{rem}

The left Kan extension of a functor $X\in \C Sp^O[O(p,q)]$ along $\alpha_{p,q}$ can be described by the following $\COTop_*$-enriched coend. 
\begin{equation*}
    ((\alpha_{p,q})_! (X))(V)=\int\limits^{U\in \Jpq} \CJ{1}{0} ((p,q)U,V)\wedge X(U).
\end{equation*}
We make this functor suitable enriched by `twisting' the action as in \cite[Definition 8.2]{BO13}. That is, we let $\C\mathcal{J}_{1,0}$ act on $\CJ{1}{0}((p,q)U,V)$ on the left by composition, and let $\C\mathcal{J}_{p,q}$ act on $\CJ{1}{0}((p,q)U,V)$ on the right by composition. It is not hard to show, by an argument of coends, that $(\alpha_{p,q})_!$ forms a left adjoint to $\alpha_{p,q}^*$. We now prove that this adjunction is indeed a Quillen equivalence. 

\begin{thm}\label{QEstabletospectra}
The adjoint pair 
\begin{equation*}
    (\alpha_{p,q})_{!}: \OEpq^s\rightleftarrows \C Sp^O [O(p,q)] : \alpha_{p,q}^*
\end{equation*}
is a Quillen equivalence, where both categories are equipped with their stable model structures (see Proposition \ref{prop: stable ms} and Theorem \ref{thm: stable ms on C2O-spectra}). 
\end{thm}

\begin{proof}
The proof follows as in \cite[Proposition 8.3]{BO13}.
Since $\alpha_{p,q}^*$ is defined by precomposition, it preserves objectwise fibrations and objectwise acyclic fibrations. It can easily be shown that $\alpha_{p,q}^*$ also preserved homotopy pullbacks. Hence, it preserves stable fibrations. An argument using finality shows that $\alpha_{p,q}^*$ preserves and reflects weak equivalences, and therefore preserves stable acyclic fibrations, making the adjunction Quillen.

It remains to show that the Quillen adjunction is a Quillen equivalence. By Hovey \cite[Theorem 1.3.16]{Hov99}, it suffices to show that the derived unit is a weak equivalence. Since the categories are stable and $\alpha_{p,q}^*$ preserves coproducts, it suffices to do this for the generators of $\OEpq$. This can be done by plugging in the generators $O(p,q)_+\wedge \pqs$ and $(\C\times O(p,q))_+\wedge \pqs$ into the formula for the unit, where $(p,q)\mathbb{S}$ sends $V\in\Jzero$ to the one point compactification of $\R^{p+q\delta}\otimes V$, denoted $S^{(p,q)V}$. 
\end{proof}

\subsection{Induction as a Quillen functor}

In this section we construct a Quillen adjunction between the $(p,q)$-homogeneous model structure and the stable model structure on $\Epq$. The right adjoint of this adjunction will be the differentiation functor $\ind_{0,0}^{p,q}\varepsilon^*$. Moreover, we will show that this Quillen adjunction is in fact a Quillen equivalence between these categories. Combined with the Quillen equivalence of Theorem \ref{QEstabletospectra}, this allows for the classification of $(p,q)$-homogeneous functors in terms of orthogonal $\C$-spectra with an action of $O(p,q)$. 

The steps taken to construct this Quillen equivalence are similar to \cite[Section 9 and Section 10]{BO13}. We begin by explicitly proving the existence of a Quillen adjunction between the projective model structures on $\OEpq$ and $\Ezero$, which can then be extended via properties of Quillen adjunctions and Bousfield localisations. 

\begin{lem}\label{Epqlevel Ezero adjunction}
    There exists a Quillen adjunction 
\begin{equation*} \res_{0,0}^{p,q}/O(p,q):\OEpql\rightleftarrows\Ezero:\ind_{0,0}^{p,q}\varepsilon^*.
\end{equation*}
\end{lem}

\begin{proof}
The generating (acyclic) cofibrations of $\OEpq^l$ are of the form 
\begin{equation*}
\Jpq(U,-)\wedge O(p,q) \wedge i,    
\end{equation*}
where $i$ is a generating (acyclic) cofibration of the fine model structure on $\CTop_*$. 

Applying the left adjoint gives $\res_{0,0}^{p,q}\Jpq(U,-)\wedge i$ in $\Ezero$. From Lemma \ref{sphereandmorpharecofibrant}, we know that the functor $\res_{0,0}^{p,q}\Jpq(U,-)$ is cofibrant in $\Ezero$, hence $\res_{0,0}^{p,q}\Jpq(U,-)\wedge i$ is a (acyclic) cofibration in $\Ezero$. Therefore, by Hovey \cite[Lemma 2.1.20]{Hov99}, the adjunction is Quillen. 
\end{proof}

This adjunction can be extended to the $(p,q)$-polynomial model structure, via a composition of Quillen adjunctions. Furthermore, since the stable model structure on $\Epq$ is a left Bousfield localisation of the projective model structure, the Theorems of Hirschhorn \cite[Theorem 3.1.6, Proposition 3.1.18]{Hir03} can be used to additionally extend the adjunction to the stable model structure. 

\begin{lem}
There exists a Quillen adjunction 
\begin{equation*} \res_{0,0}^{p,q}/O(p,q):\OEpq^s\rightleftarrows (p,q)\poly\Ezero:\ind_{0,0}^{p,q}\varepsilon^*.
\end{equation*}
\end{lem}

\begin{proof}
There exists a Quillen adjunction 
\begin{equation*}
    \Id:\Ezero\rightleftarrows (p,q)\poly\Ezero:\Id,
\end{equation*}
since the $(p,q)$-polynomial model structure is a left Bousfield localisation of the projective model structure. Since the composition of Quillen adjunctions is a Quillen adjunction, combining the above adjunction with Lemma \ref{Epqlevel Ezero adjunction} gives a Quillen adjunction 
\begin{equation*}
    \res_{0,0}^{p,q}/O(p,q):\OEpq^l\rightleftarrows (p,q)\poly\Ezero:\ind_{0,0}^{p,q}\varepsilon^*.
\end{equation*}
We now use \cite[Theorem 3.1.6, Proposition 3.1.18]{Hir03} to show that this Quillen equivalence passes to $\OEpq^s$. That is, we show that $\ind_{0,0}^{p,q}\varepsilon^*$ takes objects that are $(p,q)$-polynomial to $(p,q)\Omega$-spectra. This has been done in Proposition \ref{5.12} and Lemma \ref{omega spectra lemma}.
\end{proof}

Since the $(p,q)$-homogeneous model structure is a right Bousfield localisation of the $(p,q)$-polynomial model structure, the Theorems \cite[Theorem 3.1.6, Proposition 3.1.18]{Hir03} can again be used to extend this Quillen adjunction to the $(p,q)$-homogeneous model structure. 

\begin{lem}\label{QAstabletohomog}
There exists a Quillen adjunction 
\begin{equation*} \res_{0,0}^{p,q}/O(p,q):\OEpq^s\rightleftarrows (p,q)\homog\Ezero:\ind_{0,0}^{p,q}\varepsilon^*.
\end{equation*}
\end{lem}
\begin{proof}
Let $f:X\rightarrow Y$ be a weak equivalence between fibrant objects in the $(p,q)$-homogeneous model structure. Then the map 
\begin{equation*}
    f^*:\Nat_{0,0}(\Jpq(V,-),X)\rightarrow \Nat_{0,0}(\Jpq(V,-),Y)
\end{equation*}
is a weak equivalence of $\C$-spaces for all $V\in\Jzero$, by definition of the right Bousfield localisation. Therefore, $\ind_{0,0}^{p,q}\varepsilon^*f$ is an objectwise weak equivalence. An application of \cite[Lemma 3.1.6, Proposition 3.1.18]{Hir03} now gives that $\ind_{0,0}^{p,q}\varepsilon^*$ is right Quillen. 
\end{proof}

The Quillen adjunctions between the model categories constructed are summarised in the following diagram (which we do not claim commutes). 

% https://q.uiver.app/?q=WzAsNSxbMCwwLCJPKHAscSlDXzJcXG1hdGhjYWx7RX1fe3AscX1ebCJdLFswLDIsIk8ocCxxKUNfMlxcbWF0aGNhbHtFfV97cCxxfV5zIl0sWzIsMCwiQ18yXFxtYXRoY2Fse0V9X3swLDB9Il0sWzQsMCwiKHAscSlcXHRleHR7LXBvbHktfUNfMlxcbWF0aGNhbHtFfV97MCwwfSJdLFs0LDIsIihwLHEpXFx0ZXh0ey1ob21vZy19Q18yXFxtYXRoY2Fse0V9X3swLDB9Il0sWzEsMCwiXFx0ZXh0e2lkfSIsMix7Im9mZnNldCI6NH1dLFsyLDMsIlxcdGV4dHtpZH0iLDAseyJvZmZzZXQiOi0yfV0sWzMsMiwiXFx0ZXh0e2lkfSIsMCx7Im9mZnNldCI6LTJ9XSxbNCwzLCJcXHRleHR7aWR9IiwwLHsib2Zmc2V0IjotNH1dLFszLDQsIlxcdGV4dHtpZH0iLDAseyJvZmZzZXQiOi00fV0sWzEsNCwiXFx0ZXh0e3Jlc31fezAsMH1ee3AscX0vTyhwLHEpIiwwLHsib2Zmc2V0IjotMn1dLFs0LDEsIlxcdGV4dHtpbmR9X3swLDB9XntwLHF9XFxlcHNpbG9uXioiLDAseyJvZmZzZXQiOi0yfV0sWzAsMiwiXFx0ZXh0e3Jlc31fezAsMH1ee3AscX0vTyhwLHEpIiwwLHsib2Zmc2V0IjotMn1dLFsyLDAsIlxcdGV4dHtpbmR9X3swLDB9XntwLHF9XFxlcHNpbG9uXioiLDAseyJvZmZzZXQiOi0yfV0sWzAsMSwiXFx0ZXh0e2lkfSIsMix7Im9mZnNldCI6NH1dXQ==
\[\begin{tikzcd}
	{O(p,q)C_2\mathcal{E}_{p,q}^l} && {C_2\mathcal{E}_{0,0}} && {(p,q)\poly C_2\mathcal{E}_{0,0}} \\
	\\
	{O(p,q)C_2\mathcal{E}_{p,q}^s} &&&& {(p,q)\homog C_2\mathcal{E}_{0,0}}
	\arrow["{\Id}"', shift right=4, from=3-1, to=1-1]
	\arrow["{\Id}", shift left=2, from=1-3, to=1-5]
	\arrow["{\Id}", shift left=2, from=1-5, to=1-3]
	\arrow["{\Id}", shift left=4, from=3-5, to=1-5]
	\arrow["{\Id}", shift left=4, from=1-5, to=3-5]
	\arrow["{\res_{0,0}^{p,q}/O(p,q)}", shift left=2, from=3-1, to=3-5]
	\arrow["{\ind_{0,0}^{p,q}\varepsilon^*}", shift left=2, from=3-5, to=3-1]
	\arrow["{\res_{0,0}^{p,q}/O(p,q)}", shift left=2, from=1-1, to=1-3]
	\arrow["{\ind_{0,0}^{p,q}\varepsilon^*}", shift left=2, from=1-3, to=1-1]
	\arrow["{\Id}"', shift right=4, from=1-1, to=3-1]
\end{tikzcd}\]

\subsection{The classification of $(p,q)$-homogeneous functors}\label{sec: classification}

We now aim to show that the Quillen adjunction of Theorem \ref{QAstabletohomog}, between the $(p,q)$-stable model structure and the $(p,q)$-homogeneous model structure, is in fact a Quillen equivalence. To do this, we take the same approach as Barnes and Oman \cite[Section 10]{BO13}, but first we will need to generalise a few more results from the underlying calculus to the $\C$-equivariant setting. The following is a $\C$-generalisation of \cite[Lemma 9.3]{BO13}.

\begin{lem}\label{BO9.3}
The left derived functor of $\res_{0,0}^{p,q}/O(p,q)$ is objectwise weakly equivalent to the functor $EO(p,q)_+\wedge_{O(p,q)} \res_{0,0}^{p,q}(-)$.    
\end{lem}

\begin{proof}
 Let $X\in\OEpq^s$ and denote the cofibrant replacement of $X$ in the projective model structure $\OEpq^l$ by $\hat{c}X$. Then $\hat{c}X$ is in particular $O(p,q)$-free. Hence, there are objectwise weak equivalences
\begin{align*}
EO(p,q)_+\wedge_{O(p,q)}\res_{0,0}^{p,q}(\hat{c}X)&\rightarrow EO(p,q)_+\wedge_{O(p,q)}\res_{0,0}^{p,q}(X)\\
EO(p,q)_+\wedge_{O(p,q)}\res_{0,0}^{p,q}(\hat{c}X)&\rightarrow\res_{0,0}^{p,q}(\hat{c}X)/O(p,q)
\end{align*}
induced by the maps $\hat{c}X\rightarrow X$ and $EO(p,q)_+\rightarrow S^0$ respectively. The result follows directly from this. 
\end{proof}

The following two examples play a key role in classifying homogeneous functors. These examples generalise \cite[Example 5.7 and Example 6.4]{Wei95} respectively. 

%Weiss Example 5.7 
\begin{ex}\label{5.7}
Let $\Theta$ be an orthogonal $\C$-spectrum with $O(p,q)$-action and $p,q\geq 1$. The functor $F\in\Ezero$ defined by 
\begin{equation*}
    F:V\mapsto \Omega^\infty [(S^{(p,q)V}\wedge \Theta)_{hO(p,q)}]
\end{equation*}
is $(p,q)$-homogeneous. 
\end{ex}

\begin{proof}
Since $F$ has a delooping, by Corollary \ref{cor: F delooping}, in order to show that $F$ is $(p,q)$-polynomial it suffices to show that $F^{(p+1,q)}$ and $F^{(p,q+1)}$ are both objectwise contractible (where $F^{(m,n)}$ denotes the $(m,n)$-derivative of $F$, $\ind_{0,0}^{m,n}F$).

Recall (see Proposition \ref{loops fibre sequence}) that $F^{(p+1,q)}(V)$ is the homotopy fibre of 
\begin{equation*}
    F^{(p,q)}(V)\rightarrow \Omega^{p,q\R}F^{(p,q)}(V\oplus \R)
\end{equation*}
and that $F^{(p,q+1)}(V)$ is the homotopy fibre of 
\begin{equation*}
    F^{(p,q)}(V)\rightarrow \Omega^{p,q\Rdelta}F^{(p,q)}(V\oplus \Rdelta).
\end{equation*}
Iterating this process gives a lattice of derivatives 

\[\begin{tikzcd}
	{F^{(p,q)}} & {F^{(p-1,q)}} & \dots & {F^{(0,q)}} \\
	{F^{(p,q-1)}} & {F^{(p-1,q-1)}} \\
	\vdots && \ddots \\
	{F^{(p,0)}} &&& F
	\arrow[from=1-1, to=1-2]
	\arrow[from=1-1, to=2-1]
	\arrow[from=1-2, to=1-3]
	\arrow[from=1-3, to=1-4]
	\arrow[from=2-1, to=3-1]
	\arrow[from=3-1, to=4-1]
	\arrow[from=1-2, to=2-2]
	\arrow[from=2-1, to=2-2]
	\arrow[dotted, from=4-1, to=4-4]
	\arrow[dotted, from=1-4, to=4-4]
\end{tikzcd}\]

We will ``identify" this lattice with another lattice
% https://q.uiver.app/?q=WzAsMTAsWzAsMCwiRltwLHFdIl0sWzEsMCwiRltwLTEscV0iXSxbMiwwLCJcXGRvdHMiXSxbMywwLCJGWzAscV0iXSxbMCwxLCJGW3AscS0xXSJdLFswLDIsIlxcdmRvdHMiXSxbMCwzLCJGW3AsMF0iXSxbMywzLCJGIl0sWzEsMSwiRltwLTEscS0xXSJdLFsyLDIsIlxcZGRvdHMiXSxbMCwxXSxbMCw0XSxbMSwyXSxbMiwzXSxbNCw1XSxbNSw2XSxbMSw4XSxbNCw4XSxbNiw3LCIiLDIseyJzdHlsZSI6eyJib2R5Ijp7Im5hbWUiOiJkb3R0ZWQifX19XSxbMyw3LCIiLDAseyJzdHlsZSI6eyJib2R5Ijp7Im5hbWUiOiJkb3R0ZWQifX19XV0=
\[\begin{tikzcd}
	{F[p,q]} & {F[p-1,q]} & \dots & {F[0,q]} \\
	{F[p,q-1]} & {F[p-1,q-1]} \\
	\vdots && \ddots \\
	{F[p,0]} &&& F
	\arrow[from=1-1, to=1-2]
	\arrow[from=1-1, to=2-1]
	\arrow[from=1-2, to=1-3]
	\arrow[from=1-3, to=1-4]
	\arrow[from=2-1, to=3-1]
	\arrow[from=3-1, to=4-1]
	\arrow[from=1-2, to=2-2]
	\arrow[from=2-1, to=2-2]
	\arrow[dotted, from=4-1, to=4-4]
	\arrow[dotted, from=1-4, to=4-4]
\end{tikzcd}\]
\noindent where $F[i,j](V)=\Omega^\infty[(S^{(p,q)V}\wedge \Theta)_{hO(p-i,q-j)}]$. That is, we will verify that $F[i,j]^{(1,0)}$ is objectwise equivalent to $F[i+1,j]$ and $F[i,j]^{(0,1)}$ is objectwise equivalent to $F[i,j+1]$, as is true for the functors $F^{(i,j)}$.

Here $O(p-i,q-j)$ is the subgroup of $O(p,q)$ that fixes the first $i$ coordinates and the $(p+1)^\text{st}$ to $(p+j)^\text{th}$ coordinates. That is, for all $g$ in $O(p-i,q-j)$ and all $(x_1,...,x_{p+q})$ in $\R^{p+q\delta}$, if $g (x_1,...,x_{p+q})= (y_1,...,y_{p+q})$, then $x_n=y_n$ for all $n\leq i$ and all $p+1\leq n\leq p+j$. 

Each $F[i,j]$ is an element of $\C \mathcal{E}_{i,j}$. The structure maps are defined by the following series of maps
\begin{align*}
    S^{(i,j)U}\wedge F[i,j](V)&=S^{(i,j)U}\wedge \Omega^\infty[(S^{(p,q)V}\wedge \Theta)_{hO(p-i,q-j)}]\\
    &\rightarrow \Omega^\infty[S^{(i,j)U}\wedge(S^{(p,q)V}\wedge \Theta)_{hO(p-i,q-j)}]\\
    &=\Omega^\infty[(S^{(i,j)U}\wedge S^{(p,q)V}\wedge \Theta)_{hO(p-i,q-j)}]\\
    &\rightarrow \Omega^\infty[(S^{(p,q)U}\wedge S^{(p,q)V}\wedge \Theta)_{hO(p-i,q-j)}]\\
    &\simeq \Omega^\infty[(S^{(p,q)(U\oplus V)}\wedge \Theta)_{hO(p-i,q-j)}]\\
    &=F[i,j](U\oplus V)
\end{align*}
where the second equality holds since $O(p-i,q-j)$ fixes $\R^{i,j}$. Moreover, $F[p,q]$ is a $(p,q)\Omega$-spectrum (by substituting $F[p,q]$ into the adjoint structure map). 

We now want to show that $F[i,j]^{(1,0)}$ is objectwise equivalent to $F[i+1,j]$ and $F[i,j]^{(0,1)}$ is objectwise equivalent to $F[i,j+1]$. Then, since $F[p,q]$ is a $(p,q)\Omega$-spectrum
\begin{align*}
*\simeq F[p,q]^{(1,0)}\equiv (F^{(p,q)})^{(1,0)}=F^{(p+1,q)}\\
*\simeq F[p,q]^{(0,1)}\equiv (F^{(p,q)})^{(0,1)}=F^{(p,q+1)}
    \end{align*}
as desired. We do this by calculating $F[i,j]^{(1,0)}$ and $F[i,j]^{(0,1)}$ using Proposition \ref{loops fibre sequence}. 

First let $0\leq i<p$ and $0\leq j \leq q$. 
\begin{align*}
&F[i,j]^{(1,0)}(V)= \hofibre\left[ F[i,j](V) \rightarrow \Omega^{(i,j)\R} F[i,j](\R\oplus V) \right]\\
=& \hofibre\left[ \Omega^\infty[(S^{(p,q)V}\wedge \Theta)_{hO(p-i,q-j)}]\rightarrow \Omega^{(i,j)\R} \Omega^\infty[(S^{(p,q)(\R\oplus V)}\wedge \Theta)_{hO(p-i,q-j)}]\right] \\
=&\Omega^\infty  \hofibre\left[ (S^{(p,q)V}\wedge \Theta)_{hO(p-i,q-j)}\rightarrow \Omega^{(i,j)\R} [(S^{(p,q)(\R\oplus V)}\wedge \Theta)_{hO(p-i,q-j)}]\right] \\
\simeq & \Omega^\infty  \hofibre\left[ (S^{(p,q)V}\wedge \Theta)_{hO(p-i,q-j)}\rightarrow \Omega^{(i,j)\R} [(S^{(p,q)\R}\wedge S^{(p,q)V}\wedge \Theta)_{hO(p-i,q-j)}]\right] \\
\simeq & \Omega^\infty  \hofibre\left[ (S^{(p,q)V}\wedge \Theta)_{hO(p-i,q-j)}\rightarrow \Omega^{(i,j)\R} [(S^{(i,j)\R}\wedge S^{(p-i,q-j)\R}\wedge S^{(p,q)V}\wedge \Theta)_{hO(p-i,q-j)}]\right] \\
\simeq & \Omega^\infty  \hofibre\left[ (S^{(p,q)V}\wedge \Theta)_{hO(p-i,q-j)}\rightarrow \Omega^{(i,j)\R}\Sigma^{(i,j)\R} [(S^{(p-i,q-j)\R}\wedge S^{(p,q)V}\wedge \Theta)_{hO(p-i,q-j)}]\right] \\
\simeq &\Omega^\infty  \hofibre\left[ (S^{(p,q)V}\wedge \Theta)_{hO(p-i,q-j)}\rightarrow (S^{(p-i,q-j)\R}\wedge S^{(p,q)V}\wedge \Theta)_{hO(p-i,q-j)}\right] 
\end{align*}
where the last weak equivalence is by the $\pi_*$-equivalence $\Omega^V\Sigma^VX\rightarrow X$ for $\C Sp^O[O(p,q)]$. 

The map $S^{(p,q)V}\wedge \Theta \rightarrow S^{(p-i,q-j)\R}\wedge S^{(p,q)V}\wedge \Theta$ is given by 
\begin{equation*}
    S^{0,0}\wedge S^{(p,q)V}\wedge \Theta \xrightarrow[]{\mu\wedge \id\wedge\id} S^{(p-i,q-j)\R}\wedge S^{(p,q)V}\wedge \Theta,
\end{equation*}
where $\mu:S^{0,0}\rightarrow S^{(p-i,q-j)\R}$ is the canonical inclusion, which has stable homotopy fibre $S^{(p-i-1,q-j)\R}_+$. Therefore, the homotopy fibre of $\mu\wedge \id\wedge\id$ is $S^{(p-i-1,q-j)\R}_+\wedge S^{(p,q)V}\wedge \Theta$. Here $O(p-i,q-j)$ acts on $S^{(p-i-1,q-j)\R}_+$ by identifying $S^{(p-i-1,q-j)\R}$ with the unit sphere $S(\R^{p-i,q-j})$ in $\R^{p-i,q-j}$. Since taking homotopy orbits preserves fibre sequences, the homotopy fibre of the map $$(S^{(p,q)V}\wedge \Theta)_{hO(p-i,q-j)}\rightarrow (S^{(p-i,q-j)\R}\wedge S^{(p,q)V}\wedge \Theta)_{hO(p-i,q-j)}$$ is $(S^{(p-i-1,q-j)\R}_+\wedge S^{(p,q)V}\wedge \Theta)_{hO(p-i,q-j)}$.

Then we conclude as follows, where the second weak equivalence is described below. 
\begin{align*}
F[i,j]^{(1,0)}(V)&\simeq \Omega^\infty \left[ (S^{(p-i-1,q-j)\R}_+\wedge S^{(p,q)V}\wedge \Theta)_{hO(p-i,q-j)}\right]\\
&=\Omega^\infty \left[ (S(\R^{p-i,q-j})_+\wedge S^{(p,q)V}\wedge \Theta)_{hO(p-i,q-j)}\right]\\
&\simeq \Omega^\infty \left[ (S^{(p,q)V}\wedge \Theta)_{hO(p-i-1,q-j)}\right]\\
&=F[i+1,j]
\end{align*}

The second weak equivalence holds by Proposition \ref{sphere as quotient of orthogonal groups prop} and that for $X$ a $\C$-spectrum with $O(m,n)$-action, $m> 0$ and $n\geq 0$,
\begin{align*}
(S(\R^{m,n})_+\wedge X)_{hO(m,n)}&=EO(m,n)_+\wedge_{O(m,n)} (S(\R^{m,n})_+\wedge X)\\
&=EO(m,n)_+\wedge_{O(m,n)} (O(m,n)/O(m-1,n)_+\wedge X)\\
&=EO(m,n)_+\wedge_{O(m-1,n)} X\\
&\simeq EO(m-1,n)_+\wedge_{O(m-1,n)} X\\
&=X_{hO(m-1,n)}
\end{align*}since the map $t^*:(O(m,n)\rtimes \C)\TTop_*\rightarrow (O(m-1,n)\rtimes \C)\TTop_*$, induced by the $\C$-equivariant subgroup inclusion map $t:O(m-1,n)\rightarrow O(m,n)$, exhibits $t^*EO(m,n)$ as a model for $EO(m-1,n)$. 

Now let $0\leq i\leq p$ and $0\leq j<q$. A similar calculation shows that 
\begin{align*}
F[i,j]^{(0,1)}(V)&\simeq \Omega^\infty \left[ (S^{(p-i,q-j-1)\Rdelta}_+\wedge S^{(p,q)V}\wedge \Theta)_{hO(p-i,q-j)}\right]\\
&\simeq \Omega^\infty \left[ (S^{(p,q)V}\wedge \Theta)_{hO(p-i,q-j-1)}\right]\\
&=F[i,j+1]
\end{align*}
As above, the second weak equivalence holds by Proposition \ref{sphere as quotient of orthogonal groups prop} and that 
\begin{align*}
(S(\R^{n,m})_+\wedge X)_{hO(m,n)}&=EO(m,n)_+\wedge_{O(m,n)} (S(\R^{n,m})_+\wedge X)\\
&=EO(m,n)_+\wedge_{O(m,n)} (O(n,m)/O(n-1,m)_+\wedge X)\\
&\cong EO(m,n)_+\wedge_{O(m,n)} (O(m,n)/O(m,n-1)_+\wedge X)\\
&=EO(m,n)_+\wedge_{O(m,n-1)} X\\
&\simeq EO(m,n-1)_+\wedge_{O(m,n-1)} X\\
&=X_{hO(m,n-1)}
\end{align*}
where the third step uses the $\C$-equivariant group isomorphism $O(m,n)\rightarrow O(n,m)$ defined for $m,n\geq 1$ by 
\begin{equation*}
A\rightarrow\begin{pmatrix}
    0& \Id_n\\\Id_m &0
\end{pmatrix}A \begin{pmatrix}
    0 & \Id_m\\ \Id_n & 0
\end{pmatrix}.
\end{equation*}

What remains to show is that $F$ is $(p,q)$-reduced ($\Tpq F(V) \simeq *$ for all $V\in \Jzero$). The spectrum $X=S^{(p,q)V}\wedge \Theta$ has equivariant connectivity given by the dimension function $c^*(X)=|(p,q)V^*|+c^*(\Theta)$. That is, $\pi_n^H X$ is trivial for $n\leq c^H(X)$, where
\begin{align*}
    c^e(X)&=(p+q)\Dim(V) + c^e(\Theta) \\
    c^{\C}(X)&=p\Dim(V^{\C}) +q\Dim((V^{\C})^{\perp}) +c^{\C}(\Theta),
\end{align*}
$c^H(\Theta)$ denotes the connectivity of $\Theta$ with respect to its equivariant homotopy groups $\pi_n^H \Theta$ and $(V^{\C})^\perp$ denotes the orthogonal complement of $V^{\C}$. 

Therefore, the map $$F(V)\rightarrow *(V)=*$$ is $(c^*(X)+1)$-connected, since homotopy orbits and $\Omega^\infty$ do not decrease connectivity. Repeated application of Lemma \ref{e.3} gives that this map is a $\Tpq$-equivalence. That is, 
\begin{equation*}
    \Tpq F(V)\simeq \Tpq * (V)=*\qedhere
\end{equation*}
\end{proof}

\begin{rem}
If one replaces the group $O(p,q)$ with the group $O(p)\times O(q)$ of $\C$-equivariant linear isometries on $\mathbb{R}^{p+q\delta}$, then a generalisation like Example \ref{5.7} cannot be achieved. This is because there is no equivariant description of the sphere $S(\R^{p,q})$ as a quotient of these groups like there is for the groups $O(p,q)$ (see Proposition \ref{sphere as quotient of orthogonal groups prop}).
\end{rem}

Before we give the $\C$-generalisation of \cite[Example 6.4]{Wei95}, we first state an equivariant version of a key result used in the proof of the underlying example, the proof of which is an application of the equivariant Freudenthal suspension theorem \cite[Theorem 9.1.4]{May96}. The theorem describes how close the map $[\Omega^\infty X]_{hL} \rightarrow \Omega^\infty [X_{hL}]$ is to being an equivalence, where $X$ is a $G$-spectrum with an action of a compact Lie group $L$. 

\begin{thm}\label{loopsholimthm}
For a finite group $G$ that acts on $L$, let $X$ be a $G$-spectrum with an action of $L$, and $\R[G]$ be the regular representation of $G$. Then the canonical map 
\begin{equation*}
    \xi:[\Omega^\infty X]_{hL} \rightarrow \Omega^\infty [X_{hL}]
\end{equation*}
is $v$-connected for any dimension function $v$ satisfying 
\begin{itemize}
    \item $v(H)\leq 2c^H(X)+1\quad $
    \item $v(H)\leq c^K(X) \quad $
\end{itemize}
for all closed subgroups $H\leq G$ and subgroup pairs $K<H$, where $c^H(X):=\conn(X^H)$.
\end{thm} 

\begin{ex}\label{6.4}
Let $\Theta$ be a $\C$-spectrum with an action of $O(p,q)$, where $p,q \geq 1$. Then the input functors 
\begin{align*}
    &E:V\mapsto [\Omega^\infty(S^{(p,q)V}\wedge \Theta)]_{hO(p,q)}    \\
    &F:V\mapsto \Omega^\infty [(S^{(p,q)V}\wedge \Theta)_{hO(p,q)}]
\end{align*}
are $T_{p+1,q}T_{p,q+1}$-equivalent under the canonical subfunctor inclusion map $r:E\rightarrow F$. 
\end{ex}

\begin{proof}
The spectrum $X=S^{(p,q)V}\wedge \Theta$ has equivariant connectivity given by the dimension function $c^*(X)=|(p,q)V^*|+c^*(\Theta)$. That is, $\pi_n^H X$ is trivial for all $n\leq c^H(X)$, where
\begin{align*}
    c^e(X)&=(p+q)\Dim(V) + c^e(\Theta) \\
    c^{\C}(X)&=p\Dim(V^{\C}) +q\Dim((V^{\C})^{\perp}) +c^{\C}(\Theta),
\end{align*}
$c^H(\Theta)$ denotes the connectivity of $\Theta$ with respect to its equivariant homotopy groups $\pi_n^H \Theta$ and $(V^{\C})^\perp$ denotes the orthogonal complement of $V^{\C}$. 

Applying Theorem \ref{loopsholimthm} yields that the map $r(V):E(V)\rightarrow F(V)$ is $v$-connected, where 
\begin{align*}
    v(e)&=2(p+q)\Dim(V) + 2c^e(\Theta) +1\\
    v(\C)&=\text{min}\{2p\Dim(V^{\C}) +2q\Dim((V^{\C})^{\perp}) +2c^{\C}(\Theta)+1, (p+q)\Dim(V)+c^e(\Theta)\}
\end{align*}

Corollary \ref{erratum corollary} implies that $\tau_{p+1,q}\tau_{p,q+1}r(V):\tau_{p+1,q}\tau_{p,q+1}E(V)\rightarrow \tau_{p+1,q}\tau_{p,q+1}F(V)$ is at least $(v+1)$-connected. Repeated application of Corollary \ref{erratum corollary} yields that the connectivity of $\tau_{p+1,q}^l\tau_{p,q+1}^lr(V)$ tends to infinity as $l$ tends to infinity. Hence $T_{p+1,q}T_{p,q+1}r$ is an objectwise weak equivalence. \qedhere

\end{proof}

Theorem \ref{boclassification} below states that the Quillen adjunction of Theorem \ref{QAstabletohomog} is a Quillen equivalence. The proof resembles that of Barnes and Oman in \cite[Theorem 10.1]{BO13} and Taggart in \cite[Theorem 7.5]{Tag22unit}, using the $\C$-equivariant generalisations of \cite[Examples 5.7 and 6.4]{Wei95} and \cite[Lemma 9.3]{BO13}, which are given by Example \ref{5.7}, Example \ref{6.4} and Lemma \ref{BO9.3} respectively.

\begin{thm}\label{boclassification}
For all $p,q\geq 1$, the Quillen adjunction 
\begin{equation*} \res_{0,0}^{p,q}/O(p,q):\OEpq^s\rightleftarrows (p,q)\homog\Ezero:\ind_{0,0}^{p,q}\varepsilon^*
\end{equation*}
is a Quillen equivalence. 
\end{thm}

\begin{cor}\label{zigzagclassification}
There is an equivalence of homotopy categories 
\begin{equation*}
    \Ho(\C Sp^O[O(p,q)])\rightleftarrows \Ho((p,q)\homog\Ezero)
\end{equation*}
for $p,q\geq 1$. 
\end{cor}

Rephrasing this classification using the derived adjunctions, we can explicitly describe how $(p,q)$-homogeneous functors are completely determined by genuine orthogonal $\C$-spectra with an action of $O(p,q)$. The following classification is a $\C$-equivariant generalisation of \cite[Theorem 7.3]{Wei95}. We will denote the image $\mathbb{L}(\alpha_{p,q})_!\mathbb{R}\ind_{0,0}^{p,q}\varepsilon^*F\in \C Sp^O[O(p,q)]$ of a $F\in\Ezero$ under the derived zig-zag of Quillen equivalences by $\Theta_F^{p,q}$\index{$\Theta_F^{p,q}$}. That is, $\Theta_F^{p,q}$ is a specific $\C$-spectrum with an action of $O(p,q)$, which is determined by the functor $F$. The proof follows the method used by Taggart \cite[Theorem 8.1]{Tag22unit}.

\begin{thm}\label{weissclassification}
Let $p,q\geq 1$. If $F\in\Ezero$ is a $(p,q)$-homogeneous functor, then $F$ is objectwise weakly equivalent to 
\begin{equation*}
    V\mapsto \Omega^\infty[(S^{(p,q)V}\wedge\Theta_F^{p,q})_{hO(p,q)}]
\end{equation*}
Conversely, every functor of the form 
\begin{equation*}
    V\mapsto \Omega^\infty[(S^{(p,q)V}\wedge\Theta)_{hO(p,q)}]
\end{equation*}
where $\Theta\in \C Sp^O[O(p,q)]$ is $(p,q)$-homogeneous. 
\end{thm}
\begin{proof}
The proof of the converse statement is exactly Example \ref{5.7}. 

Let $F$ be cofibrant-fibrant in $(p,q)$-homog-$\Ezero$. That is, $F$ is $(p,q)$-homogeneous and cofibrant in the projective model structure. Define functors $E,G\in\Ezero$ by 
\begin{align*}
    E(V)&=(\ind_{0,0}^{p,q}\varepsilon^* F(V))_{hO(p,q)}\\
    G(V)&=\Omega^\infty[(S^{(p,q)V}\wedge \Theta_F^{p,q})_{hO(p,q)}]
\end{align*}
The functors $E$ and $G$ are $T_{p+1,q}T_{p,q+1}$-equivalent, since 
\begin{equation*}
    \ind_{0,0}^{p,q}\varepsilon^* F(V)=\alpha_{p,q}^*\Theta_F^{p,q}(V)\simeq\Omega^\infty(S^{(p,q)V}\wedge\Theta_F^{p,q})
\end{equation*}
and $[\Omega^\infty(S^{(p,q)V}\wedge\Theta_F^{p,q})]_{hO(p,q)}$ is $T_{p+1,q}T_{p,q+1}$-equivalent to $G$ by Example \ref{6.4}. 

Since $G$ is $(p,q)$-polynomial by Example \ref{5.7}, Lemma \ref{weiss6.3.2} implies that $G$ is objectwise weakly equivalent to $T_{p+1,q}T_{p,q+1}E$. Therefore, there is an objectwise weak equivalence between $\ind_{0,0}^{p,q}\varepsilon^*T_{p+1,q}T_{p,q+1}E$ and $\ind_{0,0}^{p,q}\varepsilon^* G$, since $\ind_{0,0}^{p,q}\varepsilon^*$ is right Quillen and preserves weak equivalences of fibrant objects. 

Using the ``identification" $\equiv$ from Example \ref{5.7}, we get the following.
\begin{equation*}
    \ind_{0,0}^{p,q}\varepsilon^*G(V)\equiv G[p,q](V):=\Omega^\infty(S^{(p,q)V}\wedge \Theta_F^{p,q})\simeq \ind_{0,0}^{p,q}\varepsilon^*F(V)
\end{equation*}
Therefore, there is a zig-zag of objectwise weak equivalences between $\ind_{0,0}^{p,q}\varepsilon^*T_{p+1,q}T_{p,q+1}E$ and $\ind_{0,0}^{p,q}\varepsilon^*F$. 

Since $F$ is $(p,q)$-homogeneous by assumption, it is in particular $(p,q)$-polynomial. Then by an application of Lemma \ref{weiss6.3.2}, there is a zig-zag of objectwise weak equivalences 
\begin{equation*}
\ind_{0,0}^{p,q}\varepsilon^*T_{p+1,q}T_{p,q+1}E\simeq\ind_{0,0}^{p,q}\varepsilon^*T_{p+1,q}T_{p,q+1}F   
\end{equation*}
That is, $E$ and $F$ are weakly equivalent in the $(p,q)$-homogeneous model structure.

Since both $E$ and $F$ are cofibrant in $(p,q)$-homog-$\Ezero$, an application of \cite[Theorem 3.2.13 (2)]{Hir03} implies that $E$ and $F$ are weakly equivalent in the $(p,q)$-polynomial model structure. Since both $E$ and $F$ are fibrant in $(p,q)$-poly-$\Ezero$, an application of \cite[Theorem 3.2.13 (1)]{Hir03} implies that $E$ and $F$ are weakly equivalent in the projective model structure. Hence, there are objectwise weak equivalences
\begin{equation*}
    G\simeq T_{p+1,q}T_{p,q+1}E\simeq T_{p+1,q}T_{p,q+1}F\simeq F
\end{equation*}
since $T_{p+1,q}T_{p,q+1}$ preserves objectwise equivalences and $F$ is $(p,q)$-homogeneous. 

For general $(p,q)$-homogeneous $F$ the result follows by cofibrantly replacing $F$ in the projective model structure and then applying the argument above. 
\end{proof}

The final theorem is an application of the classification Theorem \ref{weissclassification}. This theorem describes how the classification is actually used, in order to study one of the input functors (see Definition \ref{jzero and ezero def}). In particular, the fibre $$D_{p,q}X(V) \rightarrow T_{p+1,q}T_{p,q+1}X(V)\rightarrow T_{p,q}X(V)$$ is determined by the $(p,q)$-derivative of $D_{p,q}X$. This is analogous to studying the layers of the Taylor tower of approximations in the underlying calculus, and the statement is similar to that of Weiss \cite[Theorem 9.1]{Wei95}.
\begin{thm}
For all $X\in\Ezero$, $p,q\geq 1$ and $V\in\Jzero$, there exists homotopy fibre sequences 
\begin{equation*}
  \Omega^\infty[(S^{(p,q)V}\wedge\Theta_{D_{p,q}X}^{p,q})_{hO(p,q)}] \rightarrow T_{p+1,q}T_{p,q+1}X(V)\rightarrow T_{p,q}X(V)
\end{equation*}
\end{thm}
\begin{proof}
The map $T_{p+1,q}T_{p,q+1}X\rightarrow T_{p,q}X$ is exactly $T_{p+1,q}T_{p,q+1}$ of the canonical inclusion map $X\rightarrow T_{p,q}X$, by Lemma \ref{weiss6.3.2}. Let $D_{p,q}X$ be the homotopy fibre of this map. Then $D_{p,q}X$ is $(p,q)$-homogeneous (see Theorem \ref{thm: DYpq is homog}). An application of the classification Theorem \ref{weissclassification} gives that 
\begin{equation*}
    D_{p,q}X(V)\simeq \Omega^\infty[(S^{(p,q)V}\wedge \Theta_{D_{p,q}X}^{p,q})_{hO(p,q)}].\qedhere
\end{equation*}

\end{proof}

\bibliographystyle{alphaurl}
\bibliography{references}
\end{document}